\documentclass[11pt]{amsart}
\usepackage{amscd,amssymb,verbatim,txfonts,bbold}
\usepackage{graphicx}
\usepackage{tikz-cd}
\usepackage[capitalise]{cleveref}

\newtheorem{thm}{Theorem}[section]
\newtheorem{cor}[thm]{Corollary}
\newtheorem{lem}[thm]{Lemma}
\newtheorem{prop}[thm]{Proposition}
\newtheorem{defi}[thm]{Definition}
\newtheorem{ex}[thm]{Example}

\theoremstyle{remark}
\newtheorem{rem}[thm]{Remark}

\numberwithin{equation}{section}

\renewcommand{\O}{\varnothing}

\newcommand{\supp}{\operatorname{supp}}

\newcommand{\spn}{\operatorname{span}}

\newcommand{\range}{\operatorname{range}}
\newcommand{\Fin}{\operatorname{Fin}}

\newcommand{\N}{{\mathbb N}}
\newcommand{\R}{{\mathbb R}}

\newcommand{\cP}{{\mathcal P}}

\newcommand{\el}[1]{\ensuremath{\ell_{#1}}}
\newcommand{\Lp}[1]{\ensuremath{L_p(#1)}}
\newcommand{\Lpos}[1]{\ensuremath{\mathcal{L}_+(#1)}}
\newcommand{\Lr}[1]{\ensuremath{\mathcal{L}_r(#1)}}

\title[Band projections in spaces of regular operators]{Band projections in spaces of regular operators}

\author[D. Mu\~noz-Lahoz]{David Mu\~noz-Lahoz}
\address{Instituto de Ciencias Matem\'aticas (CSIC-UAM-UC3M-UCM)\\
Consejo Superior de Investigaciones Cient\'ificas\\
C/ Nicol\'as Cabrera, 13--15, Campus de Cantoblanco UAM\\
28049 Madrid, Spain.}
\email{davidmunozlahoz@gmail.com}

\author[P. Tradacete]{Pedro Tradacete}
\address{Instituto de Ciencias Matem\'aticas (CSIC-UAM-UC3M-UCM)\\
Consejo Superior de Investigaciones Cient\'ificas\\
C/ Nicol\'as Cabrera, 13--15, Campus de Cantoblanco UAM\\
28049 Madrid, Spain.}
\email{pedro.tradacete@icmat.es}



\begin{document}

\begin{abstract}
We introduce inner band projections in the space of regular operators on a Dedekind complete Banach lattice and study some structural properties of this class. In particular, we provide a new characterization of atomic order continuous Banach lattices as those for which all band projections in the corresponding space of regular operators are inner. We also characterize the multiplication operators $L_AR_B$ which are band projections precisely as those with $A,B$ being band projections up to a scalar multiple.
\end{abstract}

\maketitle

\section{Introduction}

If $X$ is a Banach lattice, the space of bounded linear operators $\mathcal L(X)$ can be equipped with a linear order defined by setting $S\leq T$ whenever $Sx\leq Tx$ for every $x\geq0$ in $X$. We thus refer to positive operators as those satisfying $T\geq0$. In general, the cone of positive operators does not generate the whole space $\mathcal L(X)$, so one is lead to consider the space of regular operators $\mathcal L_r(X)$ (the linear span of the positive operators.) Although this is not always a Banach lattice (see \cite{Wickstead19}), it is well known that under the assumption of $X$ being Dedekind complete, $\mathcal L_r(X)$ with the regular norm is a Banach lattice, and a Banach algebra.

Our aim in this note is to delve deeper in the relation between the algebraic and lattice structures of $\mathcal L_r(X)$. In particular, we will study bands and band projections in $\mathcal L_r(X)$ from this double perspective.

Probably, the most important band in $\mathcal L_r(X)$ is the \textit{centre}, which is the band generated by the identity operator $I_X$. The centre has been extensively studied in the literature \cite{Voigt, Wickstead76, Wickstead88, Wickstead02} (see also \cite{hui-depagter, hui-wickstead} for the extendability of some of these results to bands generated by lattice homomorphisms). Other simple band projections on $\mathcal L_r(X)$ can be obtained as follows: If $P$ and $Q$ denote band projections on the Banach lattice $X$, then the operator $T\mapsto PTQ$ also defines a band projection on $\mathcal L_r(X)$. These band projections are particularly relevant in the study of components of positive operators: if $T$ is positive, its components of the form $PTQ$, where $P$ and $Q$ are band projections on $X$, are called elementary, and are the building blocks of all components (under certain conditions, see \cite[Section 2.1]{AB2006}).

In the finite dimensional situation, where every linear operator can be identified with multiplication by a matrix, and is automatically bounded and regular, the class of band projections can be easily described: let $(e_i)_{i=1}^n$ denote a basis of pairwise disjoint positive vectors of a Banach lattice $X$, so that $\mathcal L_r(X)$ is identified with $n\times n$ matrices with respect to this basis, and for any set $A\subseteq\{1,\ldots,n\}\times\{1,\ldots,n\}$ we can define a band projection
\[
\begin{array}{cccc}
    \mathcal{P} _A\colon& \Lr X & \longrightarrow & \Lr X\\
    & (\alpha_{ij})_{i,j=1}^n&\longmapsto &(\alpha_{ij}\chi_A(i,j))_{i,j=1}^n\\
\end{array}
\]
where $\chi_A(i,j)=1$ if $(i,j)\in A$ and $\chi_A(i,j)=0$ elsewhere. Equivalently, if for $1\leq i\leq n$, $P_i$ denotes the band projection on the band generated by $e_i$ in $X$, we can write $\mathcal P_A(T)=\sum_{(i,j)\in A}P_iTP_j$. It should be clear that all band projections on $\mathcal L_r(X)$ when $X$ is finite dimensional are of the above form.

Inner band projections on $\mathcal L_r(X)$, where $X$ is a Dedekind complete Banach lattice, will be induced in a similar way as in the finite dimensional case by a band decomposition of the space: Given $\{P_\lambda\}_{\lambda\in\Lambda}$, a family of pairwise disjoint band projections on $X$, for every subset $\Gamma\subseteq \Lambda\times \Lambda$, the assignment $T\mapsto\bigvee_{(\alpha,\beta)\in\Gamma} P_\alpha TP_\beta$ for $T\geq0$ can be extended (in a unique way) to a band projection on $\mathcal L_r(X)$. Our aim here is to study this family of band projections and see how they can encode information about the underlying Banach lattice $X$.

In addition, given that the operator $T\mapsto PTQ$ defines a band projection on $\mathcal L_r(X)$ when $P$ and $Q$ are band projections on $X$, a natural question arises: can band projections on a Banach lattice $X$ be characterized by the multiplication operators they induce in $\mathcal L_r(X)$? Recall that given operators $A,B$ on a Banach space $E$, the multiplication operator $L_AR_B\colon\mathcal L(E)\rightarrow \mathcal L(E)$ is defined as $L_AR_B(T)=ATB$. In this respect, a well-known result due to Vala \cite{Vala} asserts that $L_AR_B$ is a compact operator precisely when $A$ and $B$ are compact.  Properties of multiplication operators in terms of their symbols have a long history. In particular, one can find results on spectral theory, Fredholm theory, compactness properties, computation of norms, and several other. We refer the interested reader to the survey articles contained in the proceedings volumes \cite{curto-mathieu} and \cite{mathieu}.

In the setting of regular operators, multiplication operators have been studied by Synnatzsche \cite{syn} and Wickstead \cite{Wickstead15}. With a similar philosophy to Vala's result, we will show that, assuming that $X$ is a Dedekind complete Banach lattice, the multiplication operator $L_AR_B\colon \Lr X\rightarrow \Lr X$ is a band projection precisely when $A$ and $B$ are band projections up to a scalar multiple (Theorem \ref{thm:eleop_proj}).

The paper is organized as follows: In Section \ref{sec:innerband}, we start by introducing the notion of inner band projections, which are actually meaningful due to Theorem \ref{thm:proj_general}. In fact, Subsection \ref{sec:constructioninnerband} is devoted to the details for the proof of Theorem \ref{thm:proj_general} and some examples. In Subsection \ref{sec:propertiesinnerband}, we show how the inner band projections associated to a pairwise disjoint family of band projections $\{P_\lambda\}_{\lambda\in \Lambda}$ form a Boolean algebra which is isomorphic to $2^{\Lambda\times\Lambda}$ (Proposition \ref{prop:boolean_alg}) and study some of their properties. In Section \ref{sec:allinner}, we study for which Banach lattices $X$ all band projections on $\mathcal L_r(X)$ are inner, and show that this is the case precisely when $X$ is atomic and order continuous (Theorem \ref{thm:allinner}). Finally, in Section \ref{sec:multiplication}, we characterize the multiplication operators $L_AR_B\colon\mathcal L_r(X)\rightarrow \mathcal L_r(X)$ which are band projections.

For background and terminology on Banach lattices and regular operators we refer the reader to the monographs \cite{AB2006, lindenstrauss-tzafriri, Meyer-Nieberg}.

\section{Inner band projections in $\Lr X$}\label{sec:innerband}

Given a Banach lattice $X$, recall that an \textit{ideal} $Y\subseteq X$ is a subspace which is closed under domination, in the sense that if $|x|\leq|y|$ with $y\in Y$, then $x\in Y$. A \textit{band} is an ideal which is also order closed, or, in other words, if $(y_\alpha)\subseteq Y$ and $\bigvee_\alpha y_\alpha$ exists in $X$, then this supremum belongs to $Y$. Finally, a \textit{projection band} is a band $Y$ such that $X=Y\oplus Y^\perp$, where $Y^\perp=\{x\in X:|x|\wedge|y|=0,\forall y\in Y\}.$ In this case, a positive projection onto $Y$ vanishing on $Y^\perp$ is called a \textit{band projection}. It is well known that a band $Y$ of a Banach lattice $X$ is a projection band precisely when $P_Y(x)=\bigvee\{y\in Y: 0\leq y\leq x\}$ exists in $X$ for every $x\geq0$ \cite[Proposition 1.a.10]{lindenstrauss-tzafriri}. It is also well known that $P$ is a band projection if and only if $P^2=P$ and $0\leq P\leq I_X$ \cite[Theorem 1.44]{AB2006}.

Let us denote by $\mathfrak P (X)$ (resp.\ $\mathfrak B (X)$) the Boolean algebra of band projections on $X$ (resp.\ of projection bands in $X$).

\subsection{Construction of inner band projections}\label{sec:constructioninnerband}

The goal of this section is to prove the following result.

\begin{thm}\label{thm:proj_general}
    Let $X$ be a Dedekind complete Banach lattice and let $\{P_\lambda
    \}_{\lambda \in \Lambda }$ be a family of pairwise disjoint band
    projections on $X$.
    Let $\Gamma \subseteq \Lambda \times \Lambda $ be any subset.
    Then the map
    \begin{equation}\label{eq:proj_general}
    \begin{array}{cccc}
    \mathcal{P} _\Gamma \colon& \Lpos X & \longrightarrow & \Lpos X \\
            & T & \longmapsto & \displaystyle\bigvee_{(\alpha ,\beta )\in \Gamma }
            P_\alpha TP_\beta  \\
    \end{array}
    \end{equation}
    extends to a unique
    band projection ${\mathcal{P}_\Gamma \colon \Lr X\to \Lr X}$.
\end{thm}

This theorem motivates the following definition.

\begin{defi}
    Let $X$ be a Dedekind complete Banach lattice. We say that a band
    projection $\mathcal{P}$ on $\Lr X$ is an \emph{inner band
    projection} if there exists a family of pairwise disjoint
    band projections
    $\{P_{\lambda }\}_{\lambda \in \Lambda }$ on
    $X$ such that
    $\mathcal{P}=\mathcal{P}_\Gamma $, as defined in
    \cref{thm:proj_general}, for a certain $\Gamma \subseteq
    \Lambda \times \Lambda $. We will call
    \emph{inner projection bands} the bands associated with inner band
    projections, and we will denote
    $\mathcal{B}_\Gamma =\range(\mathcal{P}_\Gamma )$.
\end{defi}

We split the proof of \cref{thm:proj_general} into several lemmas. Some of them are of interest in their own, and provide tools that will be used throughout.

\begin{rem}
    Before proceeding to the results, we need to introduce some
    notation to deal with the indices. Given $\Gamma \subseteq  \Lambda
    \times \Lambda $, where $\Lambda$ is a set, we will denote the projection onto its first
    coordinate by
    \[
    \pi _1(\Gamma)=\{\, \alpha \in \Lambda  \mid (\alpha ,\beta )\in
    \Gamma     \text{ for some }\beta \in \Lambda  \, \} ,
    \]
    and the projection onto its second coordinate by
    \[
    \pi _2(\Gamma)=\{\, \beta  \in \Lambda  \mid (\alpha ,\beta )\in
    \Gamma     \text{ for some }\alpha  \in \Lambda  \, \}
    .
    \]
    For a fixed $\alpha \in \Lambda $ we will write
    \[
    \Gamma _\alpha =\{\, \beta \in \Lambda  \mid (\alpha ,\beta )\in
    \Gamma  \, \},
    \]
    and similarly, for a fixed $\beta \in \Lambda $,
    \[
    \Gamma ^\beta  =\{\, \alpha  \in \Lambda  \mid (\alpha ,\beta )\in
    \Gamma  \, \}.
    \]
\end{rem}

\begin{lem}\label{lem:finite}
    Let $X$ be a Dedekind complete Banach lattice and let $\{P_\lambda
    \}_{\lambda \in \Lambda }$ be a family of pairwise disjoint band
    projections on $X$.
    Let $\Phi \subseteq \Lambda \times \Lambda $ be a finite subset.
    Then given $T\in \Lpos X$:
    \[
    \bigvee_{(\alpha ,\beta )\in \Phi } P_\alpha TP_\beta
    =\sum_{(\alpha ,\beta )\in \Phi }^{}P_\alpha TP_\beta .
    \]
\end{lem}
\begin{proof}
    Let $S_1, S_2 \in \Lpos X$.
    Given $x \in X_+$ and $\alpha \neq \beta $, the
    Riesz--Kantorovich formulae give
    \begin{align*}
        0\leq(P_\alpha S_1\wedge P_\beta S_2)(x)&= \inf_{0\le u\le x}
        P_\alpha (S_1(u))+P_\beta (S_2(x-u))\\
                                           &\le P_\alpha
                                           (S_1(x))\wedge P_\beta
                                           (S_2(x))=0.
    \end{align*}
    Hence $P_\alpha S_1\wedge P_\beta S_2=0$ whenever $\alpha \neq
    \beta $, for arbitrary $S_1,S_2 \in \Lpos X$. Using this, and the fact
    that the finite sum of pairwise disjoint elements is equal to their
    supremum, we have
    \begin{align}\label{eq:finite1}
    \sum_{(\alpha ,\beta )\in \Phi }^{} P_\alpha TP_\beta
    &=\sum_{\alpha \in \pi _1(\Phi )}^{}P_\alpha T\bigg(\sum_{\beta \in
        \Phi _\alpha }^{}P_\beta \bigg)
  \nonumber\\&=\bigvee_{\alpha \in \pi _1(\Phi )}^{}P_\alpha T\bigg(\sum_{\beta \in
        \Phi _\alpha }^{}P_\beta \bigg)
           \nonumber\\&=\bigvee_{\alpha \in \pi _1(\Phi )}^{}P_\alpha T\bigvee_{\beta \in
        \Phi _\alpha }^{}P_\beta.
    \end{align}
    Fixed $\alpha \in \pi _1(\Phi )$, note that for every $x \in X_+$
    we can apply again the Riesz--Kantorovich to get
    \begin{align*}
        \bigg(\bigvee_{\beta \in \Phi _\alpha } P_\alpha TP_\beta
        \bigg)(x)&=
            \sup \bigg\{\sum_{\beta \in \Phi _\alpha }^{}P_\alpha
            TP_\beta x_\beta \biggm|x_\beta \in X_+\text{ for }\beta \in \Phi
    _\alpha \text{ and }\sum_{\beta \in \Phi _\alpha
        }^{}x_\beta =x\bigg\}\\
        &=
            \sup \bigg\{P_\alpha T\sum_{\beta \in \Phi _\alpha }^{}
            P_\beta x_\beta \biggm|x_\beta \in X_+\text{ for }\beta \in \Phi
    _\alpha \text{ and }\sum_{\beta \in \Phi _\alpha
        }^{}x_\beta =x\bigg\}\\
        &\le
            P_\alpha T\Bigg(\sup \bigg\{\sum_{\beta \in \Phi _\alpha }^{}
            P_\beta x_\beta \biggm|x_\beta \in X_+\text{ for }\beta \in \Phi
    _\alpha \text{ and }\sum_{\beta \in \Phi _\alpha
        }^{}x_\beta =x\bigg\}\Bigg)\\
        &=P_\alpha T\bigg( \bigvee_{\beta \in \Phi _\alpha } P_\beta
            \bigg)(x),
    \end{align*}
    where the inequality is due to the fact that $P_\alpha T\ge 0$. Using this to continue in \eqref{eq:finite1}
    \[
    \sum_{(\alpha ,\beta )\in \Phi }^{} P_\alpha TP_\beta =\bigvee_{\alpha \in \pi
    _1(\Phi )}  P_\alpha T \bigvee_{\beta \in \Phi _\alpha } P_\beta
    \ge \bigvee_{\alpha \in \pi _1(\Phi )} \bigvee_{\beta \in \Phi
    _\alpha } P_\alpha TP_\beta =\bigvee_{(\alpha ,\beta )\in \Phi }
    P_\alpha TP_\beta .
    \]

    To see the reverse inequality, we have to invoke once again the
    Riesz--Kantorovich fomulae for finitely many operators. For $x \in X_+$, we
    have
    \begin{align*}
        \bigg( \bigvee_{\alpha \in \pi _1(\Phi )}&P_\alpha T
        \bigvee_{\beta \in \Phi _\alpha }P_\beta \bigg)
        (x)\\&=\sup \bigg\{\, \sum_{\alpha \in \pi _1(\Phi )}^{}\bigg(P_\alpha T \bigvee_{\beta \in
        \Phi _\alpha }P_\beta\bigg)(x_\alpha ) \biggm| x_{\alpha } \in
        X_+\text{ for }\alpha \in \pi _1(\Phi )\text{ and }
        \sum_{\alpha \in \pi _1(\Phi )}^{}x_\alpha =x \,
    \bigg\}
    \end{align*}
    and using that the supremum of finitely many pairwise disjoint elements is
    equal to their sum
    \begin{align*}
        \bigg( \bigvee_{\alpha \in \pi _1(\Phi )}&P_\alpha T
        \bigvee_{\beta \in \Phi _\alpha }P_\beta \bigg)
        (x)\\
        &=\sup \bigg\{\, \sum_{\alpha \in \pi _1(\Phi )}^{}P_\alpha T \sum_{\beta \in
        \Phi _\alpha }P_\beta x_\alpha \biggm|  x_{\alpha } \in
        X_+\text{ for }\alpha \in \pi _1(\Phi )\text{ and }
        \sum_{\alpha \in \pi _1(\Phi )}^{}x_\alpha =x \,
    \bigg\} \\
        &\le \sup \bigg\{\, \sum_{\alpha \in \pi _1(\Phi )}^{}P_\alpha T \sum_{\beta \in
            \Phi _\alpha }P_\beta x_{\alpha \beta } \biggm| x_{\alpha
            \beta } \in
        X_+\text{ for }(\alpha,\beta ) \in \Phi \text{ and }
        \sum_{(\alpha,\beta ) \in \Phi }^{}x_{\alpha \beta }=x\, \bigg\} \\
        &= \sup \bigg\{\, \sum_{(\alpha,\beta) \in \Phi  }^{}P_\alpha
            T P_\beta x_{\alpha \beta } \biggm| x_{\alpha
            \beta } \in
        X_+\text{ for }(\alpha,\beta ) \in \Phi \text{ and }
        \sum_{(\alpha,\beta ) \in \Phi }^{}x_{\alpha \beta }=x\, \bigg\} \\
        &=\bigg(\bigvee_{(\alpha,\beta)\in \Phi  }P_\alpha TP_\beta \bigg)(x),
    \end{align*}
    where the inequality may be seen as follows: let $x_\alpha \in
    X_+ $ for $\alpha \in \pi _1(\Phi )$ be such that
    $\sum_{\alpha \in \pi_1(\Phi )}^{}x_\alpha =x$, for each $\alpha
    \in \pi_1(\Phi )$ fix a certain
    $\beta _0(\alpha )\in \Phi _\alpha $ and define $x_{\alpha \beta
    }=P_\beta (x_\alpha )$ for $\beta \neq \beta _0(\alpha )$
    and $x_{\alpha \beta _0}=x_\alpha -\sum_{\beta \neq \beta
    _0(\alpha )}^{} P_\beta (x_\alpha )$, then
    clearly
    \[
    \sum_{(\alpha,\beta)\in \Phi  }^{}x_{\alpha \beta }=x
    \quad\text{and}\quad
    \sum_{\alpha \in \pi_1(\Phi )}^{}P_\alpha T \sum_{\beta \in \Phi _\alpha }^{}
    P_\beta (x_{\alpha \beta })=\sum_{\alpha \in \pi_1(\Phi )}^{}
    P_\alpha T \sum_{\beta \in \Phi _\alpha}^{}
    P_\beta (x_\alpha ).\qedhere
    \]
\end{proof}

Given a set $A$ we will denote by $\Fin(A)$ the family of finite
subsets of $A$. Note that $\Fin(A)$ is partially ordered by inclusion,
and with this order it is upwards directed.

\begin{lem}\label{lem:pointwise}
    Let $X$ be a Dedekind complete Banach lattice and let $\{P_\lambda
    \}_{\lambda \in \Lambda }$ be a family of pairwise disjoint band
    projections on $X$. Let $\Gamma \subseteq \Lambda \times \Lambda
    $ be any subset.
    Then given $T \in \Lpos X$ and $x \in X_+$:
    \[
    \bigg( \bigvee_{(\alpha ,\beta )\in \Gamma }
        P_\alpha TP_\beta \bigg)(x)=\sup_{\Phi \in \Fin(\Gamma )}
        \sum_{(\alpha ,\beta )\in \Phi }^{} P_\alpha TP_\beta x.
    \]
\end{lem}
\begin{proof}
    It is clear that
    \[
    \bigvee_{(\alpha ,\beta )\in \Gamma }
    P_\alpha TP_\beta =\sup_{\Phi \in \Fin(\Gamma )} \bigvee_{(\alpha
    ,\beta )\in \Phi } P_\alpha TP_\beta .
    \]
    Note that $\{\bigvee_{(\alpha ,\beta )\in \Phi } P_\alpha TP_\beta
    \}_{\Phi \in \Fin(\Gamma )}$ is a net in $\Lr X$ bounded by $T$
    and increasing, as
    \[
    \bigvee_{(\alpha ,\beta )\in \Phi_1 } P_\alpha TP_\beta\le
    \bigvee_{(\alpha ,\beta )\in \Phi_2 } P_\alpha TP_\beta
    \quad\text{whenever } \Phi _1\subseteq \Phi _2.
    \]
    Since the supremum in
    $\Lr X$ of increasing and bounded nets is taken pointwise, we have that for any given $x \in X_+$,
    \[
    \bigg(\bigvee_{(\alpha ,\beta )\in \Gamma }
    P_\alpha TP_\beta\bigg)(x) =\sup_{\Phi \in \Fin(\Gamma )} \bigg(\bigvee_{(\alpha
    ,\beta )\in \Phi } P_\alpha TP_\beta \bigg)(x)=\sup_{\Phi \in
    \Fin(\Gamma )}\sum_{(\alpha ,\beta )\in \Phi } P_\alpha
    TP_\beta x,
    \]
    where in the last equality we have applied \cref{lem:finite}.
\end{proof}

Even though the map
$\mathcal{P}_\Gamma $ is defined in
\eqref{eq:proj_general} using a supremum, it turns out to be
additive:

\begin{lem}\label{lem:additive}
    Let $X$ be a Dedekind complete Banach lattice and let $\{P_\lambda
    \}_{\lambda \in \Lambda }$ be a family of pairwise disjoint band
    projections on $X$. Let $\Gamma \subseteq  \Lambda \times \Lambda
    $ be any subset. Then
    the map $\mathcal{P}_\Gamma \colon \Lpos X\to \Lpos X$, defined in
    \eqref{eq:proj_general}, is additive.
\end{lem}
\begin{proof}
    Given $T,S \in \Lpos X$, we will check that $\mathcal{P}_\Gamma
    (T+S)=\mathcal{P}_\Gamma (T)+\mathcal{P}_\Gamma (S)$ by showing
    the two inequalities. Using that the supremum of the sum is less
    than or equal to the sum  of the suprema, we get the following
    inequality:
    \begin{align*}
        \mathcal{P}_\Gamma (T+S)&=\bigvee_{(\alpha ,\beta )\in \Gamma
        } P_\alpha (T+S)P_\beta \\
        &=\bigvee_{(\alpha ,\beta )\in \Gamma
        } (P_\alpha TP_\beta +P_\alpha SP_\beta) \\
        &\le \bigvee_{(\alpha ,\beta )\in \Gamma
        } P_\alpha TP_\beta +\bigvee_{(\alpha ,\beta )\in \Gamma } P_\alpha SP_\beta \\
        &=\mathcal{P}_\Gamma (T)+\mathcal{P}_\Gamma (S).
    \end{align*}

    For the reverse inequality, fix $\Phi _1,\Phi _2 \in \Fin(\Gamma
    )$. Using \cref{lem:finite} we have
    \begin{align*}
        \bigvee_{(\alpha ,\beta )\in \Phi _1} P_\alpha TP_\beta
        +\bigvee_{(\alpha ,\beta )\in \Phi _2} P_\alpha SP_\beta &\le
        \bigvee_{(\alpha ,\beta )\in \Phi _1 \cup \Phi _2}  P_\alpha
        TP_\beta +\bigvee_{(\alpha ,\beta )\in \Phi _1\cup \Phi
        _2}P_\alpha SP_\beta  \\
        &=\sum_{(\alpha ,\beta)\in \Phi _1\cup \Phi _2}^{}P_\alpha
        TP_\beta +\sum_{(\alpha ,\beta )\in \Phi _1\cup \Phi _2}^{}P_\alpha
        SP_\beta \\
        &=\sum_{(\alpha ,\beta )\in \Phi _1\cup \Phi _2}^{}P_\alpha
        (T+S)P_\beta \\
        &=\bigvee_{(\alpha ,\beta )\in \Phi _1\cup \Phi _2} P_\alpha
        (T+S)P_\beta \\
        &\le \mathcal{P}_\Gamma (T+S).
    \end{align*}
    Taking
    supremum over $\Phi _1 \in \Fin(\Gamma )$ in the previous
    inequality we get
    \[
    \mathcal{P}_\Gamma (T)+\bigvee_{(\alpha ,\beta ) \in \Phi _2}
    P_\alpha SP_\beta \le \mathcal{P}_\Gamma (T+S).
    \]
    But $\Phi _2 \in \Fin(\Gamma )$ is arbitrary, so taking again
    supremum over these sets we arrive at $\mathcal{P}_\Gamma
    (T)+\mathcal{P}_\Gamma (S)\le \mathcal{P}_\Gamma (T+S)$.
\end{proof}

Now we can present the proof of \cref{thm:proj_general}.

\begin{proof}[Proof of \cref{thm:proj_general}]
    We have shown in \cref{lem:additive} that the map $\mathcal{P}_\Gamma \colon \Lpos X\to \Lpos X$ is additive.
    Therefore, it extends to a unique positive map,
    which we shall denote the same way, $\mathcal{P}_\Gamma \colon \Lr
    X\to \Lr X$. From the definition of $\cP_\Gamma$ it is clear that $0\le \mathcal{P}_\Gamma (T)\le T$ for
    any $T \in \Lpos X$, so $0\le \mathcal{P}_\Gamma \le I_{\Lr X}$.
    For $\mathcal{P}_\Gamma $ to be a
    band projection it only remains to
    check that $\mathcal{P}_\Gamma \mathcal{P}_\Gamma
    =\mathcal{P}_\Gamma $.

    From $0\le \mathcal{P}_\Gamma \le I_{\Lr X}$ immediately
    follows
    $\mathcal{P}_\Gamma \mathcal{P}_\Gamma\le \mathcal{P}_\Gamma $.
    It only remains to show the other inequality. Fix $\Phi \in \Fin(\Gamma )$, then
    using Lemmas~\ref{lem:finite} and~\ref{lem:pointwise} we have that for any $T \in
    \Lpos X$ and $x \in
    X_+$:
    \begin{align*}
        \bigg(\bigvee_{(\alpha ,\beta )\in \Phi } P_\alpha \mathcal{P}_\Gamma
        (T)P_\beta\bigg)(x)&=\sum_{(\alpha ,\beta )\in \Phi
        }^{}P_\alpha \mathcal{P}_\Gamma (T)P_\beta x\\
        &=\sum_{(\alpha ,\beta )\in \Phi  }^{}P_\alpha
        \bigg(\sup_{\Psi \in \Fin(\Gamma )} \sum_{(\gamma ,\delta )\in
        \Psi }^{}P_\gamma TP_\delta P_\beta x\bigg)\\
        &=\sum_{(\alpha ,\beta )\in \Phi  }^{}P_\alpha
        \bigg(\sup_{\Psi \in \Fin(\Gamma )} \sum_{\gamma  \in \Psi
        ^{\beta  }}^{}P_\gamma TP_\beta x\bigg)\\
        &\ge \sum_{(\alpha ,\beta )\in \Phi  }^{}
        \bigg(\sup_{\Psi \in \Fin(\Gamma )} P_\alpha \sum_{\gamma  \in \Psi
        ^{\beta  }}^{}P_\gamma TP_\beta x\bigg)\\
        &=\sum_{(\alpha ,\beta )\in \Phi }^{}P_\alpha TP_\beta
        x=\bigg(\bigvee_{(\alpha ,\beta )\in \Phi } P_\alpha TP_\beta
        \bigg)(x).
    \end{align*}
    Hence $\bigvee_{(\alpha ,\beta )\in \Phi } P_\alpha
    \mathcal{P}_\Gamma (T)P_\beta \ge \bigvee_{(\alpha ,\beta )\in \Phi }
    P_\alpha TP_\beta $ for any $\Phi \in \Fin(\Gamma )$. With this we
    arrive at
    \begin{align*}
        \mathcal{P}_\Gamma \mathcal{P}_\Gamma (T)&=\bigvee_{(\alpha
        ,\beta )\in \Gamma } P_\alpha \mathcal{P}_\Gamma (T)P_\beta \\&=\sup_{\Phi \in
        \Fin(\Gamma )} \bigvee_{(\alpha ,\beta )\in \Phi } P_\alpha
        \mathcal{P}_\Gamma (T)P_\beta \\ &\ge \sup_{\Phi \in
        \Fin(\Gamma )} \bigvee_{(\alpha ,\beta )\in \Phi} P_\alpha
            TP_\beta \\ &=\mathcal{P}_\Gamma (T),
    \end{align*}
    as we wanted to show.
\end{proof}

\begin{rem}[Order continuous case]\label{rem:ordcont_inner}
    Let us check how are inner band projections when $X$ is an order continuous Banach lattice.
    Let $\{B_\lambda
    \}_{\lambda \in \Lambda }$ be a family of bands in $X$ such that
    $X=\bigoplus_{\lambda \in \Lambda }B_\lambda $; that is, any $x \in X$ can be expressed uniquely as $x=\sum_{\lambda \in \Lambda} P_\lambda x$, where $P_\lambda$ denotes the band projection onto $B_\lambda$, at most countably many of the $P_\lambda x$ are nonzero, and the series converges unconditionally. For example, if $A$ is a maximal family of disjoint positive elements of $X$, then $X=\bigoplus_{a\in A} B_a$, where $B_a$ is the band generated by $a\in A$,

    Fix $\Gamma\subseteq \Lambda\times\Lambda$. For any given $T\in\Lpos X$ and $x \in X_+$, the set \[\{\, P_\alpha TP_\beta x  \mid (\alpha ,\beta )\in \Gamma \, \}\] has at most countably many nonzero terms. Let $(a_n)_{n=1}^\infty$ be any enumeration of those terms. Using \cref{lem:pointwise} and the fact that $a_n\ge 0$,
    we get
    \begin{align*}
        \mathcal{P}_\Gamma (T)(x)&=\bigg(\bigvee_{(\alpha ,\beta )\in
        \Gamma } P_\alpha TP_\beta\bigg)(x)\\&=\sup_{\Phi \in \Fin(\Gamma )}
    \sum_{(\alpha ,\beta )\in \Phi  }^{} P_\alpha TP_\beta x\\
                                 &=\sup_{\Phi \in \Fin(\N)}
    \sum_{ n \in \Phi }^{} a_n=\sum_{n=1}^{\infty}a_n=\sum_{(\alpha
    ,\beta )\in \Gamma }^{}P_\alpha TP_\beta x.
    \end{align*}
    We can conclude that, in the order continuous case, inner band projections are computed using infinite sums.
\end{rem}

\begin{ex}
    Let $X$ be a space with a 1-unconditional basis $(e_\lambda)_{\lambda\in\Lambda}$ (with $\Lambda$ a possibly uncountable set). For
    every $\lambda  \in \Lambda $, consider the band projection
    \[
    \begin{array}{cccc}
    P_\lambda  \colon& X & \longrightarrow & X \\
                   & x & \longmapsto & x(\lambda)e_{\lambda  } \\
    \end{array}
    \]
   One can easily check that $\{P_\lambda\}_{\lambda \in \Lambda}$ is a family of pairwise disjoint band projections on $X$. So we are under the conditions of \cref{thm:proj_general}. Moreover, it is also true that $x=\sum_{\lambda \in \Lambda } P_\lambda x$ for every $x \in X$, in the sense of \cref{rem:ordcont_inner}. According to this same remark, for each $\Gamma \subseteq \Lambda \times \Lambda $ we have a band
    projection
    \[
    \begin{array}{cccc}
        \mathcal{P}_\Gamma \colon& \Lr{X} &
        \longrightarrow & \Lr{X} \\
            & T & \longmapsto & \displaystyle\sum_{(\alpha ,\beta )\in \Gamma }^{}
            P_\alpha TP_\beta \\
    \end{array}
    \]
    It will follow from \cref{thm:allinner} that these
    are all the band projections on $\Lr {X}$.
\end{ex}

\begin{ex}
    Consider the Dedekind complete Banach lattice $\el \infty (\mathbb N)$. In this
    space, the
    band projections $P_i$ onto the $i$-th coordinate satisfy the
    hypotheses of \cref{thm:proj_general}. Fix some subset $\Gamma
    \subseteq \N\times \N$. We want to illustrate with some
    examples how one can use \cref{lem:pointwise} to compute
    $\mathcal{P}_\Gamma (T)$ explicitly for certain operators $T \in \Lr{\el\infty }$.

    Fixed some $a=(a_n)_{n=1}^{\infty }\in (\el \infty )_+$, define the
    positive operator
    \[
    \begin{array}{cccc}
    T_a\colon& \el \infty  & \longrightarrow & \el \infty  \\
            & x & \longmapsto & (a_1x_1,a_2x_2,\ldots ) \\
    \end{array}
    \]
    Let us denote by $\Delta =\{\, (n,n) \mid n \in \N \, \} \subseteq \N\times \N$ the diagonal of $\N\times \N$. For any $\Phi \in \Fin(\N\times \N)$ and $x \in X_+$ we have
    \[
    \sum_{(\alpha ,\beta )\in \Phi }^{}P_\alpha TP_\beta
    x=\sum_{\alpha \in \pi _1(\Phi )}^{
    }P_\alpha T \sum_{\beta \in \Phi _\alpha }^{}P_\beta x=
    \begin{cases}
        a_nx_n&\text{ if }n \in \Delta \cap \Phi, \\
        0&\text{ if }n \not\in \Delta \cap \Phi,
    \end{cases}
    \]
    which implies, by \cref{lem:pointwise},
    \[
    \mathcal{P}_\Gamma (T_a)x=\sup_{\Phi \in \Fin(\Gamma )}
    \sum_{(\alpha ,\beta )\in \Phi }^{} P_\alpha TP_\beta x=
    \begin{cases}
        a_nx_n&\text{ if }n \in \Delta \cap \Gamma, \\
        0&\text{ if }n \not\in \Delta \cap \Gamma.
    \end{cases}
    \]
    Therefore, $\mathcal{P}_\Gamma (T_a)=T_b$, where $b_n=a_n$ if $n
    \in \Delta \cap \Gamma $, and zero otherwise.

    Now consider the shift operator
    \[
    \begin{array}{cccc}
    S\colon& \el \infty  & \longrightarrow & \el \infty  \\
            & x & \longmapsto & (x_2,x_3,\ldots ) \\
    \end{array}
    \]
    which is a positive operator. For any $\Phi \in \Fin(\N\times \N)$ and $x \in
    X_+$ we have
    \[
    \sum_{(\alpha ,\beta )\in \Phi }^{}P_\alpha SP_\beta x=
    \begin{cases}
        x_{n+1}&\text{ if }(n,n+1)\in \Phi, \\
        0&\text{ if }(n,n+1)\not\in \Phi.
    \end{cases}
    \]
    and using again \cref{lem:pointwise}
    \[
    \mathcal{P}_\Gamma (S)x=\sup_{\Phi \in \Fin(\Gamma )}
    \sum_{(\alpha ,\beta )\in \Phi }^{} P_\alpha SP_\beta x=
    \begin{cases}
        x_{n+1}&\text{ if }(n,n+1) \in \Gamma, \\
        0&\text{ if }(n,n+1) \not\in \Gamma.
    \end{cases}
    \]
\end{ex}

\subsection{Properties of inner band projections}\label{sec:propertiesinnerband} Now that we have properly defined inner band projections, we can derive some of their properties.
The following result shows that inner band projections form a Boolean algebra with the order and lattice operations inherited from the space of regular operators.

\begin{prop}\label{prop:boolean_alg}
    Let $X$ be a Dedekind complete Banach lattice and let $\{P_\lambda
    \}_{\lambda \in \Lambda }$ be a family of pairwise disjoint band
    projections on $X$. The inner band projections $\{\,
    \mathcal{P}_\Gamma  \mid \Gamma  \subseteq \Lambda
    \times \Lambda  \, \} $ (resp.\ the inner projection bands $\{\,
    \mathcal{B}_\Gamma  \mid \Gamma \subseteq \Lambda \times \Lambda  \,
    \} $) form a Boolean algebra with the same order, suprema and
    infima as in $\mathfrak{P}(\Lr X)$ (resp.\ $\mathfrak{B}(\Lr
    X)$). Moreover, this Boolean algebra is
    isomorphic to $2^{\Lambda \times \Lambda }$ through the map $\Gamma
    \mapsto \mathcal{P}_\Gamma $ (resp.\ $\Gamma \mapsto
    \mathcal{B}_\Gamma $).
\end{prop}

\begin{rem}\label{rem:noninner}
    Even though inner band projections form a Boolean algebra with the same operations as in $\mathfrak{P}(\Lr X)$, they do not form, in general, a Boolean subalgebra of $\mathfrak{P}(\Lr X)$. The reason for this is that the maximal element in the Boolean algebra of inner band projections is $\mathcal{P}_{\Lambda \times \Lambda }$, which need not be the identity, the maximal element in $\mathfrak{P}(\Lr X)$. This is true even if we assume that $\{P_\lambda\}_{\lambda\in\Lambda}$ form a decomposition of $X$, so as to have the identity $x=\bigvee_{\lambda\in\Lambda} P_\lambda x$ for every $x \in X_+$.

    Let us illustrate this with an example. Consider the Banach lattice $\el \infty$ with the disjoint band projections $\{P_n\}_{n\in \N}$, where $P_n$ is the projection corresponding to the $n$-th coordinate. We have $x=\bigvee_{n \in \N} P_n x$ for every $x \in X_+$. Let $\phi \in \el \infty^*$ be a Banach limit, and consider the positive operator $T=\phi\otimes e_1$, where $e_1=(1,0,0,\ldots)$. Then clearly $TP_n=0$, for every $n\in\N$, so $\mathcal{P}_{\N \times \N}(T)=0$, while $T\neq 0$. This implies, in particular, that $\mathcal{P}_{\N \times \N}$ is not the identity on $\Lr X$.
\end{rem}

\begin{proof}[Proof of \cref{prop:boolean_alg}]
    It suffices to show the result for the inner band projections.
    Note first that the minimal and maximal elements in the family of
    inner band projections are $\mathbb{0}=\mathcal{P}_\O$ and
    $\mathbb{1}=\mathcal{P}_{\Lambda \times \Lambda }$. As observed
    before the proposition, $\mathcal{P}_{\Lambda \times \Lambda }$
    need not be the identity, which is the maximal element of
    $\mathfrak{P}(\Lr X)$.

    Fix
    $\Gamma ,\Delta \subseteq \Lambda \times \Lambda $ and let us
    start by showing that the infimum of two inner band projections $\mathcal{P}_\Gamma
    \wedge \mathcal{P}_\Delta =\mathcal{P}_\Gamma \mathcal{P}_\Delta
    $ is given by $\mathcal{P}_{\Gamma \cap \Delta }$. For this, let $\Phi \in \Fin(\Gamma )$, $T \in \Lpos X$ and $x
    \in X_+$. Using \cref{lem:pointwise} we have
    \begin{align}\label{eq:inf_innerproj}
        \sum_{(\alpha ,\beta )\in \Phi }^{} P_\alpha
        \mathcal{P}_\Delta (T)P_\beta x&=\sum_{(\alpha ,\beta )\in
        \Phi }^{} P_\alpha \sup_{\Psi \in \Fin(\Delta )} \sum_{(\gamma
        ,\delta )\in \Psi }^{} P_\gamma TP_\delta P_\beta x\nonumber\\ \nonumber
       &= \sum_{(\alpha ,\beta )\in \Phi }^{}P_\alpha \sup_{\Psi \in
       \Fin(\Delta )} \sum_{\gamma \in \Psi ^{\beta }}^{}
            P_\gamma TP_\beta x\\
&= \sum_{(\alpha ,\beta )\in \Phi }^{}P_\alpha \bigg(\sup_{\Psi \in
       \Fin(\Delta )} \sum_{\gamma \in \Psi ^{\beta }}^{}
            P_\gamma\bigg) TP_\beta x,
    \end{align}
    where in the last identity we have used that $\{\sum_{\gamma \in
    \Psi ^{\beta }}^{} P_\gamma \}_{\Psi \in \Fin(\Delta )}$ is an
    increasing and bounded net.
    Note that
    \[
    P_\alpha \bigg(\sup_{\Psi \in
       \Fin(\Delta )} \sum_{\gamma \in \Psi ^{\beta }}^{}
            P_\gamma\bigg) =P_\alpha \bigg(\sup_{\Psi \in
       \Fin(\Delta )} \bigvee_{\gamma \in \Psi ^{\beta }}^{}
            P_\gamma\bigg) =P_\alpha \bigvee_{\gamma \in \Delta
            ^{\beta }} P_\gamma.
    \]
    Using that the projection bands $\{P_\lambda \}_{\lambda \in
    \Lambda }$ are pairwise disjoint, we can continue in
    \eqref{eq:inf_innerproj} to get
    \[
    \sum_{(\alpha ,\beta )\in \Phi }^{} P_\alpha \mathcal{P}_\Delta
    (T)P_\beta x=\sum_{\substack{(\alpha ,\beta )\in \Phi \\ \alpha
    \in \Delta ^{\beta }}}^{} P_\alpha TP_\beta x=\sum_{(\alpha ,\beta
)\in \Phi \cap \Delta }^{}P_\alpha TP_\beta .
    \]
    Applying this
    \begin{align*}
        \mathcal{P}_\Gamma (\mathcal{P}_\Delta (T))(x)&=\sup_{\Phi \in
        \Fin(\Gamma )} \sum_{(\alpha ,\beta )\in \Phi }^{}P_\alpha
        \mathcal{P}_\Delta (T)P_\beta x\\
          &=\sup_{\Phi \in \Fin(\Gamma )} \sum_{(\alpha ,\beta )\in
          \Phi \cap \Delta }^{}P_\alpha TP_\beta x\\
          &=\sup_{\Phi \in \Fin(\Gamma \cap \Delta )}\sum_{(\alpha ,\beta )\in
          \Phi }^{}P_\alpha TP_\beta x\\
          &=\mathcal{P}_{\Gamma \cap \Delta }(T)(x).
    \end{align*}

    For the computations on the complementary and the supremum, we
    need to show first that $\mathcal{P}_\Gamma +\mathcal{P}_\Delta
    =\mathcal{P}_{\Gamma \cup \Delta }$ whenever $\Gamma \cap \Delta
    =\O$. Indeed, take as usual $T \in \Lpos X$ and $x \in X_+$, then
    \begin{align*}
        (\mathcal{P}_\Gamma +\mathcal{P}_\Delta )(T)(x)&=\sup_{\Phi
        \in \Fin(\Gamma )} \sum_{(\alpha ,\beta )\in \Phi }^{}P_\alpha
        TP_\beta x+\sup_{\Psi \in \Fin(\Delta )}\sum_{(\gamma ,\delta
        )\in \Psi }^{}P_\gamma TP_\delta x\\
        &\ge \sup\bigg\{\sum_{(\alpha ,\beta )\in \Phi }^{}P_\alpha
            TP_\beta x+\sum_{(\gamma ,\delta )\in \Psi }^{}P_\gamma
            TP_\delta x\biggm| \Phi \in \Fin(\Gamma ),\, \Psi \in
        \Fin(\Delta )\bigg\}\\
        &=\sup_{\Phi \in \Fin(\Gamma \cup \Delta )} \sum_{(\alpha
        ,\beta )\in \Phi }^{}P_\alpha TP_\beta x=\mathcal{P}_{\Gamma
        \cup \Delta }(T)x,
    \end{align*}
    where we have first used \cref{lem:pointwise}, then that the sum
    of the suprema is greater than the supremum of the sum, and
    finally the disjointess of $\Gamma $ and $\Delta $ to be able to
    apply again \cref{lem:pointwise}. To see the reverse inequality,
    fix first $\Phi \in \Fin(\Gamma )$ and $\Psi \in \Fin(\Delta )$,
    then we have
    \[
        \sum_{(\alpha ,\beta )\in \Phi }^{}P_\alpha TP_\beta
        x+\sum_{(\gamma ,\delta )\in \Psi }^{}P_\gamma TP_\delta x\le
        \mathcal{P}_{\Gamma \cup \Delta }(T)x,
    \]
    where we are using that $\Phi \cap \Psi =\O$. Since these finite
    sets were chosen arbitrarily, one can take first supremum on $\Phi
    \in \Fin(\Gamma )$ and then on $\Psi \in \Fin(\Delta )$ in the
    previous expression, and thus arrive, according to
    \cref{lem:pointwise}, to
    \[
    \mathcal{P}_\Gamma (T)x+\mathcal{P}_\Delta (T)x\le
    \mathcal{P}_{\Gamma \cup \Delta }(T)x.
    \]
    We conclude that, whenever $\Gamma \cap \Delta =\O$,
    $\mathcal{P}_\Gamma +\mathcal{P}_\Delta =\mathcal{P}_{\Gamma \cup
    \Delta }$. In particular, taking $\Delta =\overline{\Gamma }=\Lambda \times \Lambda
    \setminus \Gamma$, we get
    $\mathcal{P}_{\overline{\Gamma }}=\mathcal{P}_{\Lambda \times
    \Lambda }-\mathcal{P}_\Gamma=\mathbb{1}-\mathcal{P}_\Gamma $, thus giving the appropriate formula
    for the complementary.

    Next we have to check that the supremum $\mathcal{P}_\Gamma \vee
    \mathcal{P}_\Delta =\mathcal{P}_\Gamma + \mathcal{P}_\Delta
    -\mathcal{P}_{\Gamma \cap \Delta }$ is given by
    $\mathcal{P}_{\Gamma \cup \Delta }$. We have shown this formula in
    the case that the sets are disjoint. In the general case we can
    write
    \[
        \mathcal{P}_{\Gamma \cup \Delta }=\mathcal{P}_{\Gamma \cap
        \overline{\Delta }}+\mathcal{P}_{\Delta \cap \overline{\Gamma
    }} +\mathcal{P}_{\Gamma \cap \Delta },
    \]
    since $\Gamma \cup \Delta $ is the disjoint union of the sets
    $\Gamma \cap \overline{\Delta }$, $\Delta \cap \overline{\Gamma
    }$, and $\Gamma \cap \Delta $. Now using the formulae for the
    infimum and the complementary
    \begin{align*}
        \mathcal{P}_{\Gamma \cup \Delta }=\mathcal{P}_\Gamma
        (\mathbb{1}-\mathcal{P}_\Delta )+\mathcal{P}_\Delta
        (\mathbb{1}-\mathcal{P}_\Gamma
        )+\mathcal{P}_{\Gamma}\mathcal{P}_{\Delta }=\mathcal{P}_\Gamma
        +\mathcal{P}_\Delta -\mathcal{P}_{\Gamma \cap \Delta }.
    \end{align*}

    From all these computations we can conclude that $\{\,
    \mathcal{P}_\Gamma  \mid \Gamma \subseteq \Lambda \times \Lambda
    \, \} $ forms a Boolean algebra with the same order, suprema and
    infima as in $\mathfrak{P}(\Lr X)$, and also that
    the map $\Gamma \mapsto \mathcal{P}_\Gamma $ is a Boolean algebra
    homomorphism. Clearly this map is surjective, so to establish the
    result it only remains to
    check that it is injective. Let $\Gamma ,\Delta \subseteq \Lambda
    \times \Lambda $ be different subsets; without loss of generality, we
    may assume that there exists some $(\lambda ,\mu )\in \Gamma \setminus
    \Delta $. Fix nonzero elements $x_\lambda \in (B_\lambda )_+$ and $x_\mu \in
    (B_\mu )_+$. Applying Hahn--Banach we can find
    a positive functional $x^{*}\in X^{*}_+$ such that
    $x^{*}(x_\mu )=1$. Then the positive operator $T=x_\lambda \otimes
    x^{*}$ satisfies
    \[
    \mathcal{P}_\Gamma (T)(x_\mu )=\sup_{\Phi \in \Fin(\Gamma )}
    \sum_{(\alpha ,\beta )\in \Phi }^{} P_\alpha TP_\beta x_\mu
    =x_\lambda\neq 0,
    \]
    whereas
    \[
    \mathcal{P}_\Delta (T)(x_\mu )=\sup_{\Psi  \in \Fin(\Delta )}
    \sum_{(\gamma  ,\delta  )\in \Psi }^{}P_\gamma  TP_\delta x_\mu
    =0.
    \]
    Hence $\mathcal{P}_\Gamma (T)\neq \mathcal{P}_\Delta (T)$.
\end{proof}

\begin{rem}[Order continuous case]
    In the order continuous case, inner band projections do form a Boolean subalgebra of $\mathfrak{P}(\Lr X)$, as long as the band projections $\{P_\lambda\}_{\lambda\in\Lambda}$ form a decomposition of $X$. Indeed, according to \cref{rem:ordcont_inner}, for every $T\in\Lpos X$ and $x\in X_+$ we have
    \[
    \mathcal P_{\Lambda\times\Lambda}(T)x=\sum_{(\alpha,\beta)\in\Lambda\times\Lambda}P_\alpha TP_\beta x=\sum_{\alpha\in\Lambda}P_\alpha T\sum_{\beta\in\Lambda}P_\beta x=\sum_{\alpha\in\Lambda}P_\alpha Tx=Tx.
    \]
    Hence $\mathcal P_{\Lambda\times\Lambda}(T)=T$ for every $T\in\Lpos X$, so $\mathcal P_{\Lambda\times\Lambda}=I_{\Lr X}$. This, together with the previous proposition, implies that inner band projections form a Boolean subalgebra of $\mathfrak{P}(\Lr X)$.
\end{rem}

The following is a partial characterization of inner projection bands. In the order continuous case, this result can be improved to completely characterize these bands.

\begin{prop}\label{prop:band_char}
    Let $X$ be a Dedekind complete Banach lattice and let $\{P_\lambda
    \}_{\lambda \in \Lambda }$ be a family of pairwise disjoint band
    projections on $X$. Then, for any $\Gamma \subseteq \Lambda \times
    \Lambda $, the inner projection band $\mathcal{B}_\Gamma $
    satisfies
    \[
    \mathcal{B}_\Gamma \subseteq \{\, T \in \Lr X \mid T(B_\beta
    )\subseteq B(\{\, B_\alpha  \mid \alpha \in \Gamma ^{\beta
}\})\text{ for every } \beta \in \Lambda  \, \} .
    \]
\end{prop}
\begin{proof}
Let $T \in \mathcal{B}_\Gamma $ and $x \in (B_\lambda )_+$ for a
certain $\lambda \in \Lambda $. Then, using \cref{lem:pointwise},
\[
Tx=\sup_{\Phi \in \Fin(\Gamma )} \sum_{(\alpha ,\beta )\in \Phi
}^{}P_\alpha TP_\beta x=\sup_{\Phi \in \Fin(\Gamma )} \sum_{\alpha
\in \Phi ^{\lambda }}^{} P_\alpha Tx.
\]
Note that for each given $\Phi \in \Fin(\Gamma )$, $\sum_{\alpha
\in \Phi ^{\lambda }}^{} P_\alpha Tx \in B(\{\, B_\alpha  \mid \alpha
\in \Gamma  ^{\lambda }\, \} )$, and since a band is closed under taking
arbitrary sumprema, the result follows.
\end{proof}
\begin{rem}
    The converse inclusion in \cref{prop:band_char} is
    not true in general. Consider the operator $T$ presented in
    \cref{rem:noninner}: we have $T(\spn\{e_n\})=0$ for
    each $n \in \N$. Hence, for any $\Gamma \subseteq \N\times \N$, $T$ satisfies
    \[
    T(\spn\{ e_\beta  \})\subseteq B(\{\, B_\alpha  \mid \alpha
    \in \Gamma ^{\beta } \, \}) \quad\text{for every }\beta \in \N.
    \]
    However, $T \not\in
    \mathcal{B}_\Gamma $ for any $\Gamma \subseteq \N\times \N$, because $\mathcal{P}_\Gamma(T)=0$.
\end{rem}

\begin{rem}[Order continuous case]\label{rem:band_char_ordcont}
    In the order continuous case, the other inclusion in \cref{prop:band_char} does hold, giving the following characterization of inner projection bands:
    \[
    \mathcal{B}_\Gamma =\bigg\{\, T \in \Lr X \biggm| T(B_\beta )\subseteq
    \bigoplus_{\alpha \in \Gamma ^{\beta }} B_\alpha \text{ for every }\beta
\in \Lambda   \, \bigg\}.
    \]
    Let us check the converse inclusion to \cref{prop:band_char}. Let $T\in \Lr X$ be such that $T(B_\beta )\subseteq
    \bigoplus_{\alpha \in \Gamma ^{\beta }}B_\alpha $ for every $\beta
    \in \Lambda  $. Then for any given $x \in X$,
    \[
    \mathcal{P}_\Gamma (T)x=\sum_{(\alpha ,\beta )\in \Gamma }^{}
    P_\alpha TP_\beta x=\sum_{\beta \in \pi _2(\Gamma )}^{} \bigg(
        \sum_{\alpha \in \Gamma ^{\beta }}^{} P_\alpha \bigg)TP_\beta
        x=\sum_{\beta \in \pi _2(\Gamma )}^{}TP_\beta x,
    \]
    because $P_\beta x \in B_\beta $ and by assumption $TP_\beta x \in
    \bigoplus_{\alpha \in \Gamma ^{\beta }}B_\alpha $. Now using that
    $TP_\beta x=0$ whenever $\beta \not\in \pi _2(\Gamma )$, we get
    \[
    \mathcal{P}_\Gamma (T)(x)=\sum_{\beta \in \pi _2(\Gamma )}^{}
    TP_\beta x=\sum_{\beta \in \Lambda }^{} TP_\beta
    x=T\bigg(\sum_{\beta \in \Lambda }^{} P_\beta x\bigg)=Tx.
    \]
    Hence $T=\mathcal{P}_\Gamma (T) \in \mathcal{B}_\Gamma $, as
    wanted.
\end{rem}

\begin{ex}
    Let us illustrate how one can use previous remark in a particular example.
        Let $1\le p<\infty $, let $m$ denote the Lebesgue measure on
    $[0,1]$, and let $X=\Lp m$. For a fixed $n \in \N$, consider the
    dyadic partition of the unit interval $[0,1]=\bigcup_{k=1} ^{n} I_k^{n}$, where $I_k^{n}=[(k-1)/2^{n},k/2^{n}]$ for $k=1,\ldots ,2^{n}$. Consider the band projections
    \[
    \begin{array}{cccc}
    P_k\colon& X & \longrightarrow & X \\
            & f & \longmapsto & f \chi _{I_k} \\
    \end{array},
    \]
    which have associated bands
    \[
    B_k=\range(P_k)=\{\, f \in X \mid \supp(f)\subseteq I_k \, \} .
    \]
   Clearly, $X=\bigoplus_{k=1}^{2^{n}}B_k$. Then, according to
    \cref{rem:ordcont_inner}, for any $\Gamma \subseteq \{1,\ldots
    ,2^{n}\}\times \{1,\ldots ,2^{n}\}$ the operator
    \[
    \begin{array}{cccc}
    \mathcal{P}_\Gamma \colon& \Lr X & \longrightarrow & \Lr X \\
            & T & \longmapsto & \displaystyle\sum_{(\alpha ,\beta )\in \Gamma }^{}
            P _{\alpha } T P _{\beta }\\
    \end{array}
    \]
    is a band projection.

    Set $n=2$ and $\Gamma =\{(1,1),(2,1),(3,1),(1,2)\}$. Then, according to \cref{rem:band_char_ordcont},
    $T \in \mathcal{B}_\Gamma $
    if and only if the following three conditions are satisfied:
    \begin{enumerate}
        \item $T(B_1)\subseteq B_1\oplus B_2\oplus B_3$, i.e.,
            $\supp(Tf)\subseteq [0,3/4]$ whenever $\supp(f)\subseteq
            [0,1/4]$;
        \item $T(B_2)\subseteq B_1$, i.e.,
            $\supp(Tf)\subseteq [0,1/4]$ whenever $\supp(f)\subseteq
            [1/4,1/2]$;
        \item $T(B_3)=T( B_4)=0$, i.e.,
            $Tf=0$ whenever $\supp(f)\subseteq
            [1/2,1]$.
    \end{enumerate}
    For example, the positive operator $T\colon \Lr X \to \Lr X$ given
    by
    \[
        (Tf)(x)=f(x/2) \chi _{[0,1/2]}(x)\quad\text{for every } f \in \Lr
        X\text{ and } x \in [0,1],
    \]
    satisfies the previous conditions, so $T \in \mathcal{B}_\Gamma $.

\end{ex}

For the next property we need first a definition.

\begin{defi}
    Let $X$ be a Banach space and
    let $T,S\colon X\rightarrow X$ be two operators. We say that $T$ is \emph{left
    orthogonal} to $S$ or that $S$ is \emph{right orthogonal} to $T$
    if $TS=0$. Given two families of operators
    $\mathcal{F}$ and $\mathcal{G}$ from $X$ to $X$, we say that $\mathcal{F}$
    is \emph{left orthogonal} to $\mathcal{G}$ and that $\mathcal{G}$
    is \emph{right orthogonal} to $\mathcal{F}$ if $TS=0$ for any $T
    \in \mathcal{F}$ and $S \in \mathcal{G}$.
\end{defi}

\begin{prop}
    Let $X$ be a Dedekind complete Banach lattice and let $\{P_\lambda
    \}_{\lambda \in \Lambda }$ be a family of pairwise disjoint band
    projections on $X$.
    Let $\Gamma
    ,\Delta \subseteq \Lambda \times \Lambda $. Then
    $\mathcal{B}_\Gamma  $ is left orthogonal to $\mathcal{B}_\Delta $
    if and only if $\pi _2(\Gamma )\cap \pi _1(\Delta )=\varnothing$.
\end{prop}
\begin{proof}
    Assume $\pi _2(\Gamma )\cap \pi _1(\Delta )=\varnothing$.
    Let $T \in \mathcal{B}_\Gamma $ and $S \in \mathcal{B}_\Delta $,
    we first consider the case $T,S\ge 0$. According to
    \cref{lem:pointwise}, for any $x \in X_+$:
    \[
    Sx=\sup_{\Psi \in \Fin(\Delta )} \sum_{(\gamma ,\delta )\in \Psi
    }^{} P_\gamma SP_\delta x \in B(\{\, B_\gamma  \mid \gamma \in \pi
    _1(\Delta )\, \} ),
    \]
    and since the projection bands $\{P_\lambda \}_{\lambda \in
    \Lambda }$ are pairwise disjoint
    \[
    P_\beta Sx=0\quad\text{for every }\beta \in \pi _2(\Gamma ).
    \]
    But then
    \[
    T(Sx)=\mathcal{P}_\Gamma (T)(Sx)=\sup_{\Phi \in \Fin(\Gamma )} \sum_{(\alpha ,\beta )\in \Phi
    }^{}P_\alpha TP_\beta Sx=0,
    \]
    and since $x \in X_+$ was arbitrary, it follows that $TS=0$. In
    the general case, when $T$ and $S$ are not necessarily positive,
    decomposing into positive and negative parts we get
    \[
    TS=T_+S_+ - T_-S_+ -T_+S_- +T_-S_-=0,
    \]
    because $T_+,T_- \in \mathcal{B}_\Gamma $ and $S_+,S_- \in
    \mathcal{B}_\Delta $ are all positive.

    For the reverse implication, assume $\pi _2(\Gamma )\cap \pi
    _1(\Delta )\neq \varnothing$, and let us show that
    $\mathcal{B}_\Gamma $ is not left orthogonal to
    $\mathcal{B}_\Delta $. For this we will proceed almost like in the
    order continuous case.
    Let $\gamma \in \pi _2(\Gamma )\cap \pi
    _1(\Delta )$, so that there exist $\alpha ,\beta \in \Lambda $
    with $(\alpha ,\gamma )\in \Gamma $ and $(\gamma ,\beta )\in
    \Delta $. Take nonzero elements $x_\alpha \in (B_\alpha )_+$,
    $x_\gamma \in (B_\gamma)_+ $, and $x_\beta \in (B_\beta )_+$.
    As a consequence of the classical Hahn--Banach theorem, there exist functionals $f,g \in X^{*}$ such that $f(x_\gamma
    )=1$ while $f(y)=0$ for every $y \not\in \spn\{x_\gamma \}$, and
    similarly $g(x_\beta )=1$ while $g(y)=0$ for every $y \not\in
    \spn\{x_\beta \}$.

    Define the operators $T=f\otimes x_\alpha $
    and $S=g\otimes x_\gamma $.
    Being rank one operators, $T$ and $S$ are regular. Using \cref{lem:pointwise}:
    \begin{align*}
        \mathcal{P}_\Delta (S)x_\beta &=\sup_{\Psi \in \Fin(\Delta )}
        \sum_{(\lambda ,\mu )\in \Psi }^{} P_\lambda SP_\mu x_\beta \\
                                      &=\sup_{\Psi \in \Fin(\Delta )}
        \sum_{(\lambda ,\mu )\in \Psi }^{} P_\lambda (x_\gamma
        )g(P_\mu x_\beta )\\
                                      &=\sup_{\Psi \in \Fin(\Delta )}
        \sum_{\mu \in \Psi _\gamma }^{} x_\gamma
        g(P_\mu x_\beta )\\
                                      &=\sup_{\substack{\Psi \in \Fin(\Delta )\\ (\gamma ,\beta )\in
                                      \Psi}} x_\gamma =x_\gamma
    \end{align*}
    because by assumption $(\gamma ,\beta )\in \Delta $. Repeating the
    same computations
    \[
    \mathcal{P}_\Gamma (T)x_\gamma =\sup_{\Phi \in \Fin(\Gamma )}
    \sum_{(\lambda ,\mu )\in \Phi }^{}P_\lambda (x_\alpha )f(P_\mu x_\gamma
    )=x_\alpha,
    \]
    because $(\alpha ,\gamma )\in \Gamma $. In conclusion,
    \[
    \mathcal{P}_\Gamma (T)\mathcal{P}_\Delta (S)x_\beta =x_\alpha \neq
    0,
    \]
    so $\mathcal{B}_\Gamma $ is not left orthogonal to
    $\mathcal{B}_\Delta $.
\end{proof}

\section{When are all band projections on $\Lr X$ inner?}\label{sec:allinner}
In the previous section, we have constructed a particular family of band projections on $\Lr X$, called inner band projections. It seems natural to try to characterize the Dedekind complete Banach lattices $X$ for which all band projections on $\Lr X$ are inner. This is achieved in the following theorem.

\begin{thm}\label{thm:allinner}
    Let $X$ be a Dedekind complete Banach lattice. All band projections on $\Lr X$ are inner if and only if $X$ is both atomic and order continuous.
\end{thm}

The goal of this section is to prove \cref{thm:allinner}. We do this by splitting the result in several parts: first we show that, if $X$ is not order continuous, there exists a band projection on $\Lr X$ that is not inner; then we show that, if $X$ is order continuous but not atomic, there is also a band projection on $\Lr X$ that is not inner; finally, we prove that if $X$ is atomic and order continuous, then all band projections on $\Lr X$ are inner.

Let us start by recalling a well-known separation result on Banach lattices that will be central in our proofs.

\begin{prop}\label{prop:separation}
    Let $X$ be a Banach lattice and let $I$ be a closed ideal. If $x_0
    \not\in I$, $x_0\ge 0$, there exists a positive functional $x^{*}\in
    X^{*}_+$ such that $x^{*}|_I=0$ and $x^{*}(x_0)=1$.
\end{prop}

The following procedure to construct band projections on the space of regular operators will also be very useful.

\begin{thm}
    Let $T\colon X\to
    Y$ be a positive operator between Banach lattices, with $Y$ Dedekind
    complete. For every ideal $I\subseteq X$ the formula
    \[
    T_Ix=\sup \{\, Ty \mid 0\le y\le x \text{ and } y\in I \, \} ,\quad
    \text{for } x \in X_+,
    \]
    defines a positive operator from $X$ to $Y$. Moreover, the extension of the map
    \[
    \begin{array}{cccc}
        \mathcal{P}_I\colon& \Lpos{X,Y} & \longrightarrow & \Lpos{X,Y} \\
            & T & \longmapsto & T_I \\
    \end{array}
    \]
    to the regular operators is a band projection.
\end{thm}
\begin{proof}
    See \cite[Theorems 1.28 \& 2.7]{AB2006}.
\end{proof}

Let us prove first that, when the underlying Banach lattice $X$ is not order continuous, there are band projections on $\Lr X$ that are not inner:

\begin{lem}\label{lem:not_ordcont}
    Let $X$ be a Dedekind complete Banach lattice. If $X$ is not order
    continuous, there exists a band
    projection on $\Lr X$ that is not inner.
\end{lem}
\begin{proof}
    Since $X$ is not order continuous, by \cite[Proposition
    1.a.11]{lindenstrauss-tzafriri} there
    exists a closed ideal $I$ in $X$ that is not a band. If we denote
    by
    $B_I=B(I)$ the band generated by $I$, then we are saying that $I\neq B_I$. If
    $\mathcal{P}_I$ is not inner we are done, so we may assume that
    $\mathcal{P}_I$ is inner. In this case, there exists a family of disjoint bands
    $\{B_\lambda \}_{\lambda \in \Lambda }$ and some $\Gamma \subseteq
    \Lambda \times \Lambda $ such that
    $\mathcal{P}_I=\mathcal{P}_\Gamma $; we may assume, without loss
    of generality, that the band generated by $\{B_\lambda \}_{\lambda
    \in \Lambda }$ is $X$ or, in other words, that $x=\bigvee_{\lambda
    \in \Lambda } P_\lambda x$ for all $x \in X_+$, where $P_\lambda $
    denotes the band projection
    onto $B_\lambda $.

    We claim that, for every $\lambda \in \Lambda  $, if
    $\lambda \in \pi _2(\Gamma )$ then $B_\lambda \subseteq I$, while
    if $\lambda \not\in \pi _2(\Gamma )$ then $B_\lambda \cap
    B_I=\{0\}$. Indeed, suppose that $\lambda_0 \not\in \pi _2(\Gamma
    )$
    but $B_{\lambda_0} \cap B_I\neq \{0\}$. Given $x_0\in
    B_{\lambda_0} \cap
    B_I$, $x_0>0$, we have
    \[
    \mathcal{P}_I(I_X)x_0=P_{B_I}x_0=x_0>0
    \]
    while
    \[
    0\le \mathcal{P}_\Gamma (I_X)x_0\le \bigvee_{\lambda \in \Lambda
    ,\lambda \neq \lambda _0}P_\lambda x_0=0.
    \]
    This contradicts the fact that $\mathcal{P}_I=\mathcal{P}_\Gamma
    $. Hence $B_{\lambda _0}\cap B_I=\{0\}$.

    If instead $\lambda _0=\beta _0 \in \pi _2(\Gamma )$, say $(\alpha _0,\beta
    _0) \in \Gamma $, and we suppose that $B_{\beta _0} \not \subseteq I$,
    fix $x_0\in B_{\beta _0}\setminus I$, $x_0>0$.
    By \cref{prop:separation} there exists $x^{*} \in  X^{*}_+$ such
    that $x^{*}|_I=0$ and $x^{*}(x_0)=1$. Fix also $y \in B_{\alpha
    _0}$, $y>0$, and consider the positive operator $T=x^{*}\otimes
    y$. Since $T|_I=0$, it follows that $\mathcal{P}_I(T)=T_I=0$,
    whereas
    \[
    \mathcal{P}_\Gamma (T)x_0\ge P_{\alpha _0}TP_{\beta
    _0}x_0=P_{\alpha _0}(y)x^{*}(x_0)=y> 0.
    \]
    Hence $\mathcal{P}_\Gamma (T)\neq \mathcal{P}_I(T)$, contradicting
    the assumption that $\mathcal{P}_\Gamma =\mathcal{P}_I$.

So we must have $B_\lambda  \subseteq I$ whenever $B_{\lambda
}\cap B_I\neq \{0\}$.
Since $I$ is not equal to $B_I$, by
\cref{prop:separation} there exists some $y^{*}\in X^{*}_+$ such
that $y^{*}|_I=0$ while $y^{*}(x_0)=1$ for a given $x_0\in B_I$,
$x_0>0$. Let $P\colon X\rightarrow X$ be the band projection onto $B_I$, and define the positive functional $x^*=P^*y^*$. Then $x^*$ is such that $x^{*}|_{B_I^{d}}=0$ and $x^*|_I=0$, while $x^{*}(x_0)=1$.
Pick any $y>0$ and define the positive operator $T=x^{*}\otimes y$. Clearly, we have that
$TP_\lambda =0$ for every $\lambda \in \Lambda $
because
for every $\lambda \in \Lambda $ either $B_\lambda \subseteq I$ or
$B_\lambda \cap B_I=\{0\}$, which implies $B_\lambda \subseteq
B_I^{d}$.

We claim that the band generated by $T$ in $\Lr X$, $\mathcal{B}_T$,
is not inner. To check this,
    let $\{C_\xi \}_{\xi \in \Xi }$ be any family of
    disjoint bands in $X$, and let $\{Q_\xi \}_{\xi  \in \Xi }$ be the
    corresponding band projections. Suppose that the projection onto
    $\mathcal{B}_T$ is $\mathcal{P}_\Delta $ for some $\Delta
    \subseteq \Xi \times \Xi $. Since $\mathcal{B}_T$ is not the zero
    band, $\Delta $ is not empty.
    Fix some $(\gamma  _0,\delta  _0)\in \Delta $. Since the family
    $\{B_\lambda \}_{\lambda \in \Lambda }$ forms a decomposition of
    $X$, there exist $\lambda _1,\lambda _2 \in \Lambda $ such that
    $D=B_{\lambda _1}\cap C_{\delta  _0}\neq \{0\}$ and $D'=B_{\lambda
    _2}\cap C_{\gamma  _0}\neq \{0\}$. Fix $y_0\in D$, $y_0>0$, by
    \cref{prop:separation} there exists $f \in X^{*}_+$ such that
    $f|_{D^{d}}=0$ and $f(y_0)=1$. Fix $z_0 \in D'$, $z_0>0$, and
    consider the positive operator $S=f\otimes z_0$. Then for every $x
    \in X_+$:
    \[
        (T\wedge S)(x)=\inf_{0\le u\le x} Tu+S(x-u)=0,\quad
        \text{taking }u=P_{\lambda _1}Q_{\delta  _0}x,
    \]
    so $S \in \mathcal{B}_T^{d}$. However
    \[
    \mathcal{P}_\Delta  (S)(y_0)\ge Q_{\gamma  _0}SQ_{\delta
    _0}y_0=Q_{\gamma _0}z_0=z_0>0,
    \]
    which implies $\mathcal{P}_\Delta(S)\neq  0$, a contradiction with
    the fact that $S \in \mathcal{B}_T^{d}$. This contradiction arises
    from the assumption that $\mathcal{P}_\Delta $ is the band projection
    onto $\mathcal{B}_T$. Therefore, the band projection onto
    $\mathcal{B}_T$ cannot be inner.
\end{proof}

When the Banach lattice $X$ is order continuous, but not atomic, we can still find band projections on the space of regular operators that are not inner.

\begin{lem}\label{lem:not_atomic}
    Let $X$ be a Dedekind complete Banach lattice. If $X$ is order
    continuous and not atomic, there exists a band projection on $\Lr
    X$ that is not inner.
\end{lem}
\begin{proof}
    Since $X$ is Dedekind complete, the center of $X$ is the band
    generated by $I_X$ in $\Lr X$ (see \cite[Theorem
    3.31]{abramovich-aliprantis}). Let $\mathcal{P}_Z$ be the band
    projection onto the center. We are going to show that
    $\mathcal{P}_Z$ is not inner.

    Let $\{B_\lambda \}_{\lambda \in
    \Lambda }$ be any decomposition of $X$ into disjoint bands. Since
    $X$ is not atomic, there exists some $\lambda _0\in \Lambda $ such
    that $B_{\lambda _0}$ is not atomic. We may write
    \[
    B_{\lambda _0}=B_{\lambda _0}^{1}\oplus B_{\lambda _0}^{2},
    \]
    where $B_{\lambda _0}^{2}$ is the band generated by the atoms
    contained in $B_{\lambda _0}$ and $B_{\lambda _0}^{1}$ is its
    complementary; that is, $B_{\lambda _0}^{1}$ does not contain any
    atom (of $X$, since atoms relative to an ideal are also atoms in
    the whole space). By assumption, $B_{\lambda _0}^{1}\neq \{0\}$, so
    we can pick some $e \in B_{\lambda _0}^{1}$, $e>0$. Then $B_e$ is
    a band in $X$ containing no atoms and let $P_e$ denote the corresponding band projection.

    Since $B_e$ is an order continuous Banach lattice with weak unit
    $e$, by \cite[Theorem 1.b.14]{lindenstrauss-tzafriri} there exists
    a probability space $(\Omega ,\Sigma ,\mu )$, a (non
    necessarily closed) ideal $\tilde X$ in $L_1(\Omega ,\Sigma ,\mu
    )$ and a lattice norm $\|{\cdot }\|_{\tilde X}$ in $\tilde X$ such that $B_e$ is
    order isometric to $(\tilde X, \|{\cdot }\|_{\tilde X})$ and we
    have continuous and dense inclusions
    \[
    L_\infty (\Omega ,\Sigma ,\mu )\hookrightarrow \tilde X
    \hookrightarrow L_1(\Omega ,\Sigma ,\mu ).
    \]
    Note that, if $(\Omega ,\Sigma ,\mu )$ had an atomic measurable
    set $A \in \Sigma $, then $\chi _A \in L_\infty (\Omega ,\Sigma
    ,\mu )\subseteq \tilde X$ would be an atom, contradicting the fact
    that $B_e$ has no atoms.

    Consider the positive operator
    \[
    \begin{array}{cccc}
    T\colon& L_1(\Omega ,\Sigma ,\mu ) & \longrightarrow & L_1(\Omega
    ,\Sigma ,\mu ) \\
            & f & \longmapsto & \int_{\Omega }^{} f\, d \mu \; \chi
            _\Omega   \\
    \end{array}.
    \]
    Clearly, we can restrict $T\colon \tilde X\to \tilde X$ and,
    by means of the aobve-mentioned order isometry, we can think
    of this operator as a positive operator $T\colon B_e\to B_e$.
    Define $S=TP_e\colon X\to X$. We claim now that
    $\mathcal{P}_\Gamma \neq \mathcal{P}_Z$ for every $\Gamma
    \subseteq \Lambda \times \Lambda $. Clearly, if $(\lambda
    _0,\lambda _0) \not\in \Gamma $, $\mathcal{P}_\Gamma (P_e)=0$
    whereas $\mathcal{P}_Z(P_e)=P_e$ (recall that $0\le P_e\le I_X$).
    So we must have $(\lambda _0,\lambda _0)\in \Gamma $, in which
    case $\mathcal{P}_\Gamma (S)=S$; we are going to show that
    $\mathcal{P}_Z(S)=0$, thus completing the argument.

    For this, we will again think of $T$ as an operator acting on
    $\tilde X$ and in the same way we will identify $I_{B_e}$ with
    $I_{\tilde X}$. Applying the
    Riesz--Kantorovich formulae
    we have that for every $A \in \Sigma $:
    \begin{align*}
        (T\wedge I_{\tilde X})(\chi _A)&=\inf_{0\le y\le \chi  _A}
        Ty+(\chi
        _A-y)\\
                                  &\le \inf_{B\subseteq A, B \in \Sigma } T \chi
                                  _{B} + \chi _{A\setminus B}\\
                                  &=\inf_{B\subseteq A, B \in \Sigma
                                  }\mu (B) \chi _\Omega +\chi
                                  _{A\setminus B}.
    \end{align*}
    We have already observed that the measure space $(\Omega, \Sigma, \mu)$ has no atoms, so for
    every $n \in \N$ there exist pairwise disjoint $B_1,\ldots ,B_n \in \Sigma $ such
    that
    \[
    A=\bigcup_{j=1}^{n} B_j,\quad\text{with }\mu (B_j)=\frac{\mu
    (A)}{n}\text{ for }j=1,\ldots ,n.
    \]
    Applying these sets to the previous inequality gives
    \[
        (T\wedge I_{\tilde X})( \chi _A)\le \bigwedge_{j=1} ^{n} \left(\frac{\mu
        (A)}{n} \chi _{\Omega } + \chi _{A\setminus
    B_j}\right)=\frac{\mu (A)}{n} \chi _{\Omega }.
    \]
    Since this inequality holds for every $n \in \N$, it follows that
    $(T\wedge I_{\tilde X})(\chi _A)=0$ for every $A \in \Sigma $. But the linear span of the
    characteristic functions is dense in $L_\infty (\Omega
    ,\Sigma ,\mu )$, which is in turn dense in $\tilde X$, so $T\wedge
    I_{\tilde X}$ is zero on $\tilde X$. Going back to the space
    $X$, this means that $T \wedge I_{B_e}$ is zero on $B_e$; but $S$ is
    zero outside $B_e$, and is $T$ when restricted to $B_e$, so
    $S\wedge I_X=0$. In conclusion,
    $\mathcal{P}_Z(S)=0$, and this finishes the proof.
\end{proof}

To complete the proof of \cref{thm:allinner} it only remains to show that, if the underlying Banach lattice $X$ is atomic and order continuous, then all the band projections on $\Lr X$ are inner.

Let $X$ be a Dedekind complete Banach lattice and let $a \in X$ be an atom. The band generated by $a$ is $B_a=\spn\{a\}$. Let $P_a$ be the projection onto $B_a$. Then for every $x\in X$ we can then write $P_a(x)=\lambda_a(x)a$, where $\lambda_a\colon X\rightarrow \R$ is a positive functional. Banach lattices that are both atomic and order continuous are characterized by the following well known property (see \cite[Chapter 5]{Troitsky}).

\begin{thm}
    Let $X$ be an atomic Banach lattice and let $A$ be a maximal family of pairwise disjoint atoms in $X$. Then $X$ is order continuous if and only if
    \[
    x=\sum_{a \in A} \lambda_a(x)a
    \]
    for every $x\in X_+$, where for each $x$, the series has at most countably many nonzero terms and converges unconditionally.
\end{thm}

The following is the natural generalization of elementary matrices on $M_n(\R)$ (i.e., those with one in a given entry and zeros everywhere else) to general atomic Banach lattices.

\begin{defi}
    Let $X$ be an atomic Banach lattice and let $A$ be a maximal
    family of disjoint atoms. We call \emph{elementary operators} the
    operators of the form $E_{ab}=\lambda _b\otimes a$,
    where $a,b \in A$.
\end{defi}

Elementary operators have the following properties.

\begin{lem}\label{lem:eleop}
    Let $X$ be an atomic and Dedekind complete Banach lattice, and let
    $A$ be a maximal family of disjoint atoms.
    \begin{enumerate}
        \item Elementary operators are atoms of $\Lr X$.
        \item If $\mathcal{P}\colon \Lr X\to \Lr X$ is a band
            projection,
            $\mathcal{P}(E_{ab})$ is either $E_{ab}$ or $0$ for any
            $a,b \in A$.
    \end{enumerate}
\end{lem}
\begin{proof}
    \begin{enumerate}
        \item Given $a,b \in A$, by definition $a\ge 0$ and also
            $\lambda _b\ge 0$, so $E_{ab}=\lambda _b\otimes a\ge 0$.
         Let $0\le T\le E_{ab}$. For every $x \in X_+$ we have
            \begin{align*}
                0\le Tx - TP_bx&=T\bigg(\bigvee_{\alpha \in A}
                P_\alpha x-P_bx\bigg)\\
                               &=T\bigg(\bigvee_{\alpha \in A}
                               (P_\alpha x-P_bx)\bigg)\\
                               &\le T\bigg(\bigvee_{\alpha \in A,\alpha\neq  b}
                               P_\alpha x\bigg)\\
                               &\le E_{ab}\bigg(\bigvee_{\alpha\in A, \alpha\neq b}
                               P_\alpha x\bigg)=\lambda _{b}\bigg(\bigvee_{\alpha
                                   \in A,\alpha\neq b}
                               P_\alpha x\bigg)a=0.
            \end{align*}
            Therefore $Tx=T(P_bx)=\lambda _b(x)Tb$. But $0\le Tb\le
            E_{ab}(b)=a$, and since $a$ is an atom, it follows that
            $Tb=\gamma a$ for a certain $\gamma \in \R_+$. So
            \[
            Tx=\lambda _b(x)\gamma a=\gamma E_{ab}(x)\quad\text{for
            every }x \in X_+,
            \]
            that is, $T=\gamma E_{ab}$.
        \item Since $\mathcal P$ is a band projection, $0\le
            \mathcal{P}(E_{ab})\le E_{ab}$ for any $a,b \in A$. Since
            elementary operators are atoms,
            $\mathcal{P}(E_{ab})=\gamma
            E_{ab}$ for a certain $\gamma \in \R_+$. Now using
            $\mathcal{P}^2=\mathcal{P}$,
            \[
            \gamma
            E_{ab}=\mathcal{P}(E_{ab})=\mathcal{P}(\mathcal{P}(E_{ab}))=\gamma
            ^2E_{ab},
            \]
            which implies that either $\gamma =0$ or $\gamma
            =1$.\qedhere
    \end{enumerate}
\end{proof}

\begin{rem}\label{rem:bandproj_eleop}
    Under the hypotheses of previous lemma, fix $a,b \in A$. Then for any $T \in \Lpos X$
    and $x \in X_+$ we have:
    \[
    P_aTP_bx=\lambda _b(x) P_a(Tb)=\lambda _b(x)\lambda
    _a(Tb)a=\lambda _a(Tb)E_{ab}x.
    \]
    Hence $P_aTP_b=\lambda _a(Tb)E_{ab}$. Previous lemma
    implies that, whenever $\mathcal{P}(E_{ab})=E_{ab}$, $\mathcal{P}(P_aTP_b)=P_aTP_b$ for every $T
    \in \Lpos X$ (and therefore for every $T \in \Lr X$), whereas, if
    $\mathcal{P}(E_{ab})=0$, then
    $\mathcal{P}(P_aTP_b)=0$ for every $T \in \Lr X$.
\end{rem}

\begin{lem}\label{lem:atomic_ordcont}
    Let $X$ be an atomic and order continuous Banach lattice. Then every band projection on $\Lr X$ is inner.
\end{lem}
\begin{proof}
    Let $A$ be a maximal family of pairwise disjoint atoms in $X$. Let $P_a$ be the band projection onto the band generated by $a$. We are going to consider inner band projections with respect to the family $\{P_a\}_{a \in A}$.
    For a given
    band projection $\mathcal{P}\colon \Lr X\to \Lr X$, define
    \[
    \Gamma =\{\, (a,b)\in A\times A \mid \mathcal{P}(E_{ab})=E_{ab} \,
    \}.
    \]
    Note that, using \cref{rem:bandproj_eleop}, we have for any $T \in
    \Lpos X$:
    \[
    \mathcal{P}(\mathcal{P}_\Gamma (T))=\mathcal{P}\bigg(
        \bigvee_{(a,b)\in \Gamma } P_aTP_b\bigg)\ge \bigvee_{(a,b)\in
    \Gamma } \mathcal{P}(P_aTP_b)=\bigvee_{(a,b)\in \Gamma }
    P_aTP_b=\mathcal{P}_\Gamma (T).
    \]
        Hence $\mathcal{P}\wedge \mathcal{P}_\Gamma
    =\mathcal{P}\mathcal{P}_\Gamma \ge \mathcal{P}_\Gamma $, which
    implies $\mathcal{P}_\Gamma \mathcal{P} =\mathcal{P}_\Gamma $ and also $\mathcal{P} \ge  \mathcal{P}_\Gamma $.
    We are going to show that $\mathcal{P}=\mathcal{P}_\Gamma $. Suppose, on the contrary, that they were different. Then there
    would exist some $S \in \Lpos X$ such that
    $\mathcal{P}(S)>\mathcal{P}_\Gamma (S)=\mathcal{P}_\Gamma
    (\mathcal{P}(S))$; in other words, $T=\mathcal{P}(S)\in
    \range(\mathcal{P})$ satisfies $T\ge 0$ and $T-\mathcal{P}_\Gamma
    (T)>0$. Let $x \in X_+$ be such that $(T- \mathcal{P}_\Gamma (T))x>0$,
    then by continuity and linearity of the operators
    \[
        (T-\mathcal{P}_\Gamma (T))x=\sum_{a \in A}^{}
        (T-\mathcal{P}_\Gamma (T))P_ax>0.
    \]
    It follows that there exists some $b \in A$ such that
    $(T-\mathcal{P}_\Gamma (T))P_bx>0$. This, in turn, implies that there is
    some $a \in A$ for which $P_a(T-\mathcal{P}_\Gamma (T))P_bx>0$.
    Hence
    \begin{equation}\label{eq:atordcont_ineq}
        P_a(T-\mathcal{P}_\Gamma (T))P_b>0.
    \end{equation}

    If $(a,b)\in \Gamma $, by definition of the inner band projections
    $P_a \mathcal{P}_ \Gamma (T)P_b=P_aTP_b$, contradicting
    \eqref{eq:atordcont_ineq}. Therefore $(a,b)\not\in \Gamma $, so
    $P_a \mathcal{P}_\Gamma (T)P_b=0$, which by
    \eqref{eq:atordcont_ineq} gives $P_aTP_b>0$. But since
    $0<P_aTP_b\le T$, $P_aTP_b$ is also an element of the band
    $\range( \mathcal{P})$, which implies
    $\mathcal{P}(P_aTP_b)=P_aTP_b>0$. According to the remark after
    \cref{lem:eleop},
    this is equivalent to $\mathcal{P}(E_{ab})=E_{ab}$, which contradicts
    the definition of $\Gamma $. This contradiction arises from
    assuming that $\mathcal{P}$ and $\mathcal{P}_\Gamma $ are
    different, so we must have
    $\mathcal{P}=\mathcal{P}_\Gamma $.
\end{proof}

From Lemmas~\ref{lem:not_ordcont}, \ref{lem:not_atomic} and~\ref{lem:atomic_ordcont} it immediately follows \cref{thm:allinner}. We end this section with a couple of consequences of \cref{thm:allinner} (although, to be more precise, these are only consequences of \cref{lem:atomic_ordcont}).

\begin{cor}
    Let $X$ be an atomic and order continuous Banach lattice. Let $A$
    be a maximal family of disjoint atoms in $X$. Then we have an
    isomorphism of Boolean algebras
    \[
    \begin{array}{cccc}
    & 2^{A\times A} & \longrightarrow & \mathfrak{P}(\Lr X) \\
            & \Gamma  & \longmapsto & \mathcal{P}_\Gamma  \\
    \end{array}
    \]
\end{cor}
\begin{proof}
    By \cref{lem:atomic_ordcont}, $\mathfrak{P}(\Lr X)$ contains only
    the inner band projections that arise from the band projections
    $\{P_a\}_{a \in A}$ on $X$. According to \cref{prop:boolean_alg}, the
    family of inner band projections forms a Boolean algebra that is
    isomorphic to $2^{A\times A}$ through the isomorphism given.
\end{proof}

\begin{cor}
    Let $X$ and $Y$ be atomic and order continuous Banach lattices.
    Let $A$ and $B$ be maximal families of disjoint atoms in $X$ and
    $Y$, respectively. If $|A|\neq |B|$, then the Banach lattice
    algebras $\Lr X$ and $\Lr Y$ are not isomorphic as Banach lattices.
\end{cor}
\begin{proof}
    This follows from the fact that the Boolean algebra formed by band projections is invariant under lattice isomorphisms.
\end{proof}

\section{The multiplication operator $L_AR_B$}\label{sec:multiplication}
Multiplication operators are defined in the following way.
\begin{defi}
Let $X_1,X_2,X_3,X_4$ be Banach spaces. Fixed $A \in
\mathcal{L}(X_1,X_2)$ and $B \in \mathcal{L}(X_3,X_4)$ we define the \emph{multiplication operator}
\[
\begin{array}{cccc}
    L_AR_B\colon& \mathcal{L}(X_4,X_1) & \longrightarrow & \mathcal{L}(X_3,X_2) \\
        & T & \longmapsto & ATB \\
\end{array}.
\]
\end{defi}
These operators have been an extensive object of research in different settings. Here we are going to characterize, assuming that $X=X_1=X_2=X_3=X_4$ is a Dedekind complete Banach lattice, when the multiplication operator $L_AR_B\colon \Lr X\rightarrow \Lr X$ is a band projection.

Our results will frequently rely on adaptations of the following
observation,
which can be found in \cite{mathieu-tradacete}.

\begin{rem}
    Let $X_1,X_2,X_3,X_4$ be Banach spaces and let $A \in
    \mathcal{L}(X_1,X_2)$ and $B \in \mathcal{L}(X_3,X_4)$ be nonzero
    bounded operators. Consider
    the multiplication operator $L_AR_B\colon\mathcal{L}(X_4,X_1)\to
    \mathcal{L}(X_3,X_2)$ given by $L_AR_B(T)=ATB$. Choose $x_1\in
    X_1$, $x_2^{*} \in X_2^{*}$, $x_3 \in X_3$ and $x_4^{*} \in
    X_4^{*}$ such that $x_2^{*}(Ax_1)=1=x_4^{*}(Bx_3)$. Then the
    operators
    \[
    \begin{array}{cccc}
    J_{x_4^{*}}\colon & X_1 & \longrightarrow & \mathcal{L}(X_4,X_1) \\
        & x & \longmapsto & x_4^{*}\otimes x \\
    \end{array},\quad
    \begin{array}{cccc}
    J_{x_1}\colon & X_4^{*} & \longrightarrow & \mathcal{L}(X_4,X_1) \\
        & x^{*} & \longmapsto & x^{*}\otimes x_1 \\
    \end{array},
    \]
    \[
    \begin{array}{cccc}
    \delta _{x_3}\colon & \mathcal{L}(X_3,X_2) & \longrightarrow & X_2 \\
        & T & \longmapsto & Tx_3 \\
    \end{array},\quad
    \begin{array}{cccc}
    \delta _{x_2^{*}}\colon & \mathcal{L}(X_3,X_2) & \longrightarrow &
    X_3^{*} \\
        & T & \longmapsto & T^{*}x_2^{*} \\
    \end{array},
    \]
    make the following diagrams commute
    \begin{equation}\label{eq:mult_factorization}
    \begin{tikzcd}
        X_1\arrow[r, "A"]\arrow[d, "J_{x_4^{*}}" left]& X_2\\
        \mathcal{L}(X_4,X_1)\arrow[r, "L_AR_B"]&
        \mathcal{L}(X_3,X_2)\arrow[u, "\delta _{x_3}" right]
    \end{tikzcd}\quad
    \begin{tikzcd}
        X_4^{*}\arrow[r, "B^{*}"]\arrow[d, "J_{x_1}" left]& X_3^{*}\\
        \mathcal{L}(X_4,X_1)\arrow[r, "L_AR_B"]&
        \mathcal{L}(X_3,X_2)\arrow[u, "\delta _{x_2^{*}}" right]
    \end{tikzcd}.
    \end{equation}
\end{rem}

Let us start by characterizing when $L_AR_B$ is positive.

\begin{prop}\label{prop:multoppos}
    Let $X_1,X_2,X_3,X_4$ be Dedekind complete Banach lattices and let $A \in
    \mathcal{L}(X_1,X_2)$ and $B \in \mathcal{L}(X_3,X_4)$ be nonzero
    operators. Then $L_AR_B$ is a positive operator if and only if $A$
    and $B$ are both positive or both negative.
\end{prop}
\begin{proof}
    If $A$ and $B$ are both positive or negative, then clearly
    $L_AR_B$ is positive.

    To see the converse, suppose first that both $A$ and $B$ are not negative. Then there exist $x_3 \in (X_3)_+$ and $x_1
    \in (X_1)_+$ such that $(Bx_3)_+\neq 0$ and $(Ax_1)_+\neq 0$. By
    the Hahn--Banach theorem for positive functionals there exists
    $x_2^{*} \in (X_2^{*})_+$ such that
    \[
    x_2^{*}((Ax_1)_+)=1\text{ and }x_2^{*}((Ax_1)_-)=0,\text{ that is,
    } x_2^{*}(Ax_1)=1;
    \]
    and there exists $x_4^{*} \in (X_4^{*})_+$ such that
    \[
    x_4^{*}((Bx_3)_+)=1\text{ and }x_4^{*}((Bx_3)_-)=0,\text{ that is,
    } x_4^{*}(Bx_3)=1.
    \]
    It immediately follows that the operators $J_{x_4^{*}}, J_{x_1}, \delta
    _{x_3}, \delta _{x_2^{*}}$ are all positive. Since the composition of positive operators is
    positive, by \eqref{eq:mult_factorization} we have $A\ge 0$ and
    $B^{*}\ge 0$, which implies $B\ge 0$.

    Now suppose that $A$ is negative. Take $x_1\in (X_1)_+$ such that
    $Ax_1\neq 0$, and let $-x_2^{*}\in (X_2)^{*}_+$ be such that
    $x_2^{*}(Ax_1)=1$. Now it is clear that $\delta _{x_2^{*}}\le 0$,
    while $J_{x_1}\ge 0$, so $B^{*}=\delta _{x_2^{*}}L_AR_B J_{x_1}\le
    0$, and this again implies $B\le 0$. Similarly, if $B$ is negative, take
    $x_3 \in (X_3)_+$ such that $Bx_3\neq 0$, then there exists
    $-x_4^{*} \in (X_4^{*})_+$ with $x_4^{*}(Bx_3)=1$. It follows that
    $J_{x_4^{*}}\le 0$ while $\delta _{x_3}\ge 0$ so $A=\delta
    _{x_3}L_AR_B J_{x_4^{*}}\le 0$.
\end{proof}

The previous result is enunciated but not proved in \cite[Satz
2.2]{syn}, although with the further assumption that $X_4$ is
\emph{total}.

If $A,B \in \Lr X$ are positive projections, it is not
difficult to check that $L_AR_B$ is a positive projection. Naturally,
one can ask whether the converse is true. Note that if $\lambda \in \R\setminus\{0\}$, $L_{\lambda
A}R_{B/\lambda }=L_AR_B$, but $\lambda A$ and $B/\lambda $ are no
longer positive projections (in general). It turns out, however, that
this is the only obstruction.

\begin{prop}
    Let $X$ be a Dedekind complete Banach lattice, and let $A,B \in
    \mathcal{L}(X)$ be nonzero operators. The multiplication operator
    $L_AR_B\colon\Lr X\to \Lr X$ is a positive projection
    if and only if there exists
    $\lambda \in \R\setminus\{0\}$ such that $ \lambda A$ and $B/\lambda $ are
    positive projections.
\end{prop}
\begin{proof}
    Suppose first that $\lambda A$ and $B/\lambda $ are positive
    projections for a certain $\lambda \in \R\setminus\{0\}$. Then $L_AR_B(T)=(\lambda A)T(B/\lambda )\ge 0$ for
    any $T\in \Lpos X$ and
    \begin{align*}
        L_AR_B(L_AR_B(T))&= AL_AR_B(T) B=A^2TB^2\\&=(\lambda A)^2T(
        B/\lambda )^2=(\lambda A)T(B/\lambda )=L_AR_B(T).
    \end{align*}
    Hence $L_AR_B$ is a positive projection.

    Conversely, if $L_AR_B$ is a positive projection, by the previous
    proposition $A$ and $B$ are either both positive or negative. Assume
    first they are both positive. Suppose there exists some $x \in
    X_+$ such that $Bx$ and $B^2x$ are linearly independent. Then for
    any $y\in X_+$ there is a $T\in \Lr X$ such that $T(Bx)=y$ and
    $T(B^2x)=y$ (just take $x^{*} \in X^{*}$ such that
    $x^{*}(Bx)=x^{*}(B^2x)=1$, which exists by Hahn--Banach, and set
    $T=x^{*}\otimes y$). The fact that $A^2TB^2x=ATBx$ implies $A^2y=Ay$, so
    $A=A^2$. Since $A$ is not constant (it is a nonzero operator),
    there are some $y,z\ge 0 $ such that $Ay\neq Az$, but then
    choosing $T\in \Lr X$ such that $T(Bx)=y$ and $ T(B^2x)=z$ contradicts
    the fact that $ATBx=ATB^2x$.

    So for each $x \in X_+$ with $Bx\neq 0$ there exists some real
    $\lambda (x)\ge 0$ satisfying $B^2x=\lambda (x)Bx$; in fact, since
    $Bx=0$ readily implies $B^2x=0$, we may assume that this relation
    holds for any $x \in X$, setting $\lambda (x)$ to be any positive
    number whenever $Bx=0$. For any $x,y \in X_+$
    we have the following relation:
    \[
    \lambda (x+y)(Bx+By)=\lambda (x+y)B(x+y)=B^2(x+y)=B^2x+B^2y=\lambda (x)Bx+\lambda
    (y)By.
    \]
    Now we can distinguish two different possibilities:
    \begin{enumerate}
        \item If $Bx$ and $By$ are linearly independent, the previous
            relation readily implies $\lambda (x)=\lambda
            (x+y)=\lambda (y)$.

        \item If $Bx$ and $By$ are linearly dependent, there exist $\mu
            _1,\mu _2 \in \R_+$, not both zero, such that $\mu _1 Bx+\mu _2 By=0$.
            Applying $B$ to this relation we get
            \[
            \mu _1 \lambda (x) Bx+\mu _2\lambda (y)By=0.
            \]
            If $\mu _1\neq 0$, $Bx=-\mu _2/\mu _1 By$, and dividing
            by $\mu _1$ the previous equation
            \[
            \lambda (x) Bx - \lambda (y)Bx=0,
            \]
            so $\lambda (x)=\lambda (y)$ whenever $Bx\neq 0$. If $\mu
            _2\neq 0$, similarly $\lambda (x)=\lambda (y)$ whenever
            $By\neq 0$.
    \end{enumerate}
    So $\lambda (x)=\lambda $ is constant for any $x \in X_+$ with
    $Bx\neq 0$. We conclude that $B^2=\lambda B$, i.e.,
    $B/\lambda $ is a positive projection. Then
    $\lambda A^2TB=ATB$, and given $x \in X_+$ with $Bx\neq 0$, for any
    $y\ge 0$, choosing $T \in \Lr X$ such that $T(Bx)=y$ we get from
    $\lambda A^2TBx=ATBx$ that $\lambda A^2y=Ay$. Therefore
    $A^2=A/\lambda $ and thus $\lambda A$ is a positive projection.

    When $A$ and $B$ are both negative, applying the previous argument to
    $L_{-A}R_{-B}=L_AR_B$ we conclude that there exists some $\mu
    \ge 0$ such that $(-B)^2=\mu  (-B)$ and $(-A)^2=1/\mu
    (-A)$, so the constant in this case is $\lambda =-\mu $.
\end{proof}

If in the previous proposition we replace positive projections for band
projections, everything works.

\begin{thm}\label{thm:eleop_proj}
    Let $X$ be a Dedekind complete Banach lattice, and let $A,B \in
    \mathcal{L}(X)$ be nonzero operators. The multiplication
    operator $L_AR_B\colon\Lr X\to \Lr X$ is a band projection
    if and only if there exists
    $\lambda \in \R\setminus\{0\}$ such that $\lambda A$ and $B/\lambda $ are
    band projections.
\end{thm}
\begin{proof}
    If $\lambda A$ and $B/\lambda $ are band projections, then
    $L_AR_B=L_{\lambda A}R_{B/\lambda }$ is a positive projection
    as we have already seen. If $T\in \Lpos X$:
    \[
    L_AR_B(T)=(\lambda A)T(B/\lambda )\le I_X T I_X=T,
    \]
    which proves $L_AR_B\le I_{\Lr X}$. It follows that $L_AR_B$ is a band projection.

    Conversely, if $L_AR_B$ is a band projection, by the previous
    proposition there exists $\lambda \in \R\setminus\{0\}$ such that $P=\lambda A$ and
    $Q=B/\lambda  $ are positive projections. Also, for any $T \in
    \Lpos X$,
    \begin{equation}\label{eq:multopband}
    PTQ=L_AR_B(T)\le I_{\Lr X}(T)=T
    \end{equation}
    so, multiplying by $Q$ on the right, $PTQ\le TQ$. This means that
    the multiplication operator $L_{I_X-P}R_Q(T)=(I_X-P)TQ=TQ-PTQ$ is positive, and by
    \cref{prop:multoppos} this implies that $I_X-P\ge 0$ (because
    $Q\ge 0$). So $0\le P\le I_X$ and $P^2=P$, that is, $P$ is a band
    projection. In the same way, multiplying \eqref{eq:multopband} on
    the left by $P$ we get $PTQ\le PT$, which implies that
    $L_{P}R_{I_X-Q}(T)=PT-PTQ\ge 0$. Applying
    \cref{prop:multoppos} we conclude that also $Q\le I_X$, that is,
    $Q$ is also a band projection.
\end{proof}

\section*{Acknowledgements}

Research of D. Muñoz-Lahoz supported by JAE Intro ICU scholarship associated to CEX2019-000904-S funded by MCIN/AEI/10.13039/501100011033. Research of P. Tradacete partially supported by grants PID2020-116398GB-I00 and CEX2019-000904-S funded by MCIN/AEI/10.13039/501100011033, as well as by a 2022 Leonardo Grant for Researchers and Cultural Creators, BBVA Foundation.

\end{document}